\documentclass[reqno]{amsart}
\usepackage[foot]{amsaddr}
\usepackage{graphicx} %
\usepackage{graphicx} %
\usepackage{amsmath,amssymb,amsthm}
\usepackage{algorithm}
\usepackage{algorithmic}
\usepackage{multicol}
\usepackage{fullpage}
\usepackage{hyperref}
\usepackage{xcolor}
\usepackage{mleftright}
\usepackage[shortlabels]{enumitem}
\usepackage{caption}
\usepackage{float}
\usepackage{hhline}
\usepackage{tikz}
\usepackage{subcaption}
\usepackage{tikz-cd}
\usepackage{quiver} 
\usepackage[nameinlink]{cleveref}
\usepackage{refcount}

\usepackage{mathtools}

\newcommand{\qand}{\quad\text{and}\quad}
\newcommand{\qwhere}{\quad\text{where}\quad}

\newcommand{\A}{{\mathcal{A}}}
\newcommand{\B}{{\mathcal{B}}}

\newcommand{\C}{{\mathbb{C}}}
\newcommand{\R}{{\mathbb{R}}}

\newcommand{\N}{{\mathbb{N}}}

\renewcommand{\H}{{\mathcal{H}}}

\newcommand{\eps}{\varepsilon}

\renewcommand{\Re}{\mathrm{Re}}
\renewcommand{\Im}{\mathrm{Im}}

\newcommand{\wrt}{\,\textnormal d}

\newcommand{\wt}{\widetilde}
\newcommand{\conj}{\overline}

\newcommand{\abs}[1]{\mleft|#1\mright|}
\newcommand{\sabs}[1]{|#1|}
\newcommand{\magn}[1]{\left\|#1\right\|}
\newcommand{\smagn}[1]{\|#1\|}
\newcommand{\pare}[1]{\mleft(#1\mright)}

\newcommand{\sqbrac}[1]{{\mleft[{#1}\mright]}}

\newcommand{\set}[1]{{\left\{{#1}\right\}}}
\newcommand{\bmat}[1]{\begin{bmatrix}#1\end{bmatrix}}
\newcommand{\pmat}[1]{\begin{pmatrix}#1\end{pmatrix}}
\newcommand{\alg}[1]{\textnormal{\texttt{#1}}}

\newcommand{\floor}[1]{\mleft\lfloor#1\mright\rfloor}

\newcommand{\spliteq}[2]{\begin{equation}#1\begin{split}#2\end{split}\end{equation}}
\newcommand{\eq}[1]{\begin{equation}{#1}\end{equation}}

\DeclareMathOperator*{\E}{\mathbb{E}}

\DeclareMathOperator{\vol}{Vol}

\DeclareMathOperator{\conv}{conv}

\DeclareMathOperator{\disk}{Disk}

\DeclareMathOperator{\dist}{dist}

\newcommand{\one}{\mathbf1}
\newcommand{\toth}{^{\textnormal{th}}}

\newtheorem{theorem}{Theorem}[section]
\newtheorem{lemma}[theorem]{Lemma}
\newtheorem{proposition}[theorem]{Proposition}
\newtheorem{corollary}[theorem]{Corollary}

\AddToHook{env/lemma/begin}{\crefalias{theorem}{lemma}}
\AddToHook{env/corollary/begin}{\crefalias{theorem}{corollary}}
\AddToHook{env/definition/begin}{\crefalias{theorem}{definition}}
\AddToHook{env/proposition/begin}{\crefalias{theorem}{proposition}}
\crefname{theorem}{Theorem}{Theorems}
\crefname{lemma}{Lemma}{Lemmas}
\crefname{equation}{}{}

\newcommand{\at}[1]{^{(#1)}}

\makeatletter
\newtheorem*{rep@theorem}{\rep@title}
\newcommand{\newreptheorem}[2]{%
\newenvironment{rep#1}[1]{%
 \def\rep@title{#2 \ref{##1}}%
 \begin{rep@theorem}}%
 {\end{rep@theorem}}}
\makeatother

\newreptheorem{theorem}{Theorem}
\newreptheorem{coro}{Corollary}

\newtheorem{remark}[theorem]{Remark}

\theoremstyle{definition}
\newtheorem{definition}[theorem]{Definition}

\title[The Petaloid Law]{Spectra of Random Polynomial Matrices: the Petaloid Law}
\author{Rikhav Shah$^\dagger$}
\address{$^\dagger$Massachusetts Institute of Technology}
\email{rdshah@mit.edu}
\author{Edward Zeng$^\ddagger$}
\address{$^\ddagger$New York University}
\email{ezeng@nyu.edu}
\date{\today}

\newcommand{\oned}{{1\textnormal D}}
\newcommand{\twod}{{2\textnormal D}}
\newcommand{\zerod}{{0\textnormal D}}
\newcommand{\inn}{\textnormal{in}}
\newcommand{\out}{\textnormal{out}}
\newcommand{\Cone}{\textnormal{C1} }

\newcommand{\calE}{\mathcal E}
\newcommand{\shape}{\mathcal A}

\newcommand{\Lloc}[1]{L_{\textnormal{loc}}^#1}

\begin{document}

\maketitle

\begin{abstract}

We study the distribution of the zeros of \(\det P_N(z)\) where $P_N(z)$ is a random monic polynomial matrix, i.e., \(P_N(z)=z^dI-\sum_{j=0}^{d-1}A_{j,N}z^j\) for possibly coupled random matrices $A_{j,N}$, scaled to have entrywise variance $O(1/N)$. We provide general conditions under which this distribution almost surely converges weakly as $N\to\infty$ to a deterministic measure, depending on just the variance and covariance of the entries of the $A_j$.

This generalizes the circular, elliptic, and semicircle laws, which concern the special case of this question where $d=1$.
Unlike those classical laws, these measures can combine nonuniform two-dimensional densities with singular components supported on curves, producing a variety of petal-shaped regions, inspiring our name ``the petaloid law''. We give explicit formulas for the limiting densities and supports.
Under a Gaussianity assumption, we also show that there are almost surely no eigenvalues outside small neighborhoods of the limiting support. 
\end{abstract}

\section{Introduction}

Given a matrix $A\in\C^{N\times N}$, the standard eigenvalue problem (SEP) asks for the zeros of the polynomial $\det(zI-A)$. A central problem in the field of random matrix theory has been understanding the distribution of the eigenvalues of a random matrix ensemble $A$, particularly in the limit of $N\to\infty$. More precisely, if $\set{\lambda_j}_{j=1}^N$ is the multiset of $N$ eigenvalues of $A$, then the measure
\eq{\label{a0}\mu_A(\B)=\frac1N\#\set{ j\in[N]:\lambda_j\in\B }}
denotes the empirical spectral density (ESD) of $A$.
Two of the cornerstone results in this area are Wigner's semi-circular law \cite{b0} and Girko's circular law \cite{b1}. The former concerns matrices $A$ whose entries above the diagonal are independent copies of a centered random variable with variance $1/N$, and entries below the diagonal are their conjugates so that $A=A^*$. In the large-$N$ limit, the ESD converges to the \textit{semi-circular law}, a distribution on the real interval $[-2,2]$ with density function proportional to $\sqrt{4-z^2}$. When the entries of $A$ are complex Gaussian in this setting, $A$ is called a GUE matrix. The latter concerns matrices $A$ where all entries are independent copies of a centered random variable with variance $1/N$. Since these matrices are not Hermitian, the eigenvalues are complex and in the large-$N$ limit, the ESD converges to the \textit{circular law}, the uniform distribution on the disk of radius 1. When the entries of $A$ are complex Gaussian in this setting, $A$ is called a Ginibre matrix. Pictorially, these distributions can be visualized in the scatter plots and histogram in \Cref{a1}.
\begin{figure}[ht]
    \centering

    \begin{subfigure}[t]{0.3\textwidth}
        \centering
        \includegraphics[width=\linewidth]{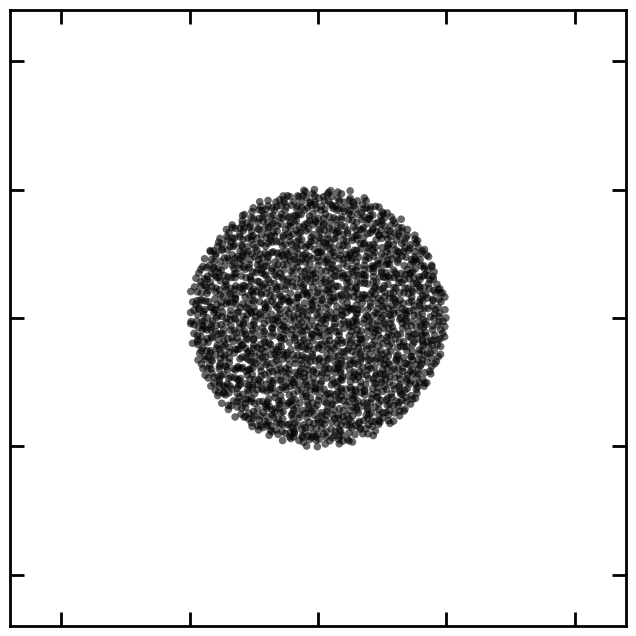}
        \caption{Scatter plot of the eigenvalues of a $2000\times2000$ Ginibre matrix.}
        \label{a2}
    \end{subfigure}
    \hspace{2mm}
    \begin{subfigure}[t]{0.3\textwidth}
        \centering
        \includegraphics[width=\linewidth]{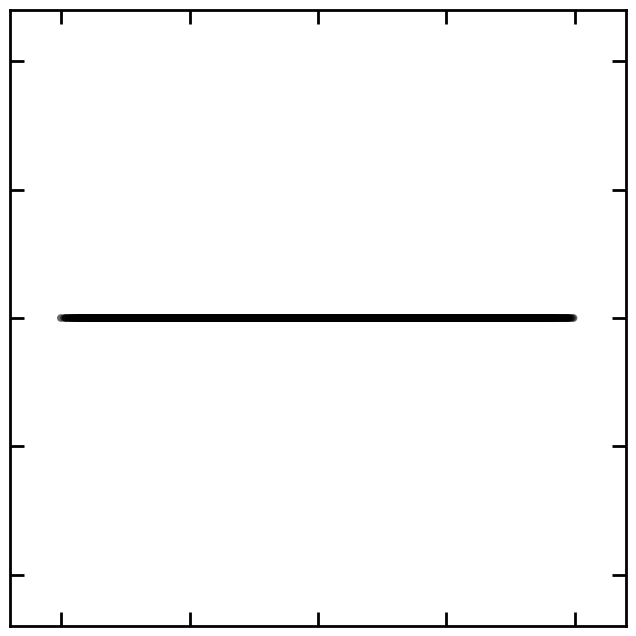}
        \caption{Scatter plot of the eigenvalues of a $2000\times2000$ GUE matrix.}\label{a3}
    \end{subfigure}
    \hspace{2mm}
    \begin{subfigure}[t]{0.3\textwidth}
        \centering
            \raisebox{0.5\height}{%
        \includegraphics[width=\linewidth]{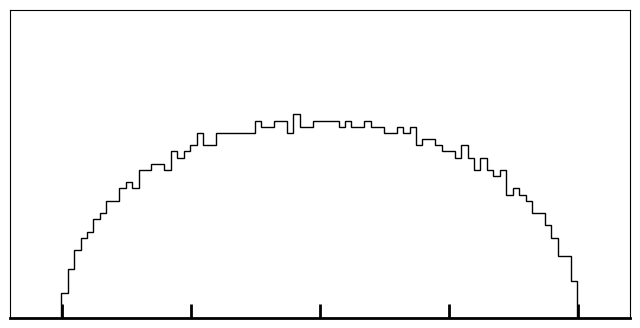}}
        \caption{Histogram of the eigenvalues of a $2000\times2000$ GUE matrix.}
        \label{a4}
    \end{subfigure}
    \caption{Demonstration of the circular and semi-circular laws.}
    \label{a1}
\end{figure}

A natural generalization of the SEP that has received less attention is the \textit{polynomial eigenvalue problem} (PEP).
This generalization replaces $zI-A$ with a higher-degree polynomial matrix,
\[P(z)=z^dI-A_{d-1}z^{d-1}-\cdots-A_1z-A_0\in\C[z]^{N\times N}\]
for a tuple of matrices $A_0,\ldots,A_{d-1}\in\C^{N\times N}$. The relevant quantity remains the roots of the scalar polynomial $\det(P(z))$, which are the called the (generalized) eigenvalues of $P$. To be explicit about the terminology, when $d=1$ and $P(z)=zI-A_0$, the standard eigenvalues of the matrix $A_0$ are precisely the generalized eigenvalues of the polynomial matrix $P$.
This work focuses on the monic case where the leading coefficient is $I$, though this may also be replaced with any matrix. The PEP, especially the $d=2$ case where it is called the \textit{quadratic eigenvalue problem}, appears in a variety of contexts in physics and engineering, e.g. mass-spring-damper models, acoustics, control theory, model order reduction, etc \cite{b2,b3,b4}. Corresponding to \Cref{a0}, define the ESD of $P$ as the uniform distribution on the $Nd$ eigenvalues $\set{\lambda_j}_{j=1}^{Nd}$ of $P$,
\eq{\label{a5}
\mu_P(\mathcal B)=\frac1{Nd}\#\set{j\in[Nd]:\lambda_j\in\mathcal B}
}
This work seeks the large-$N$ limit of $\mu_P$ for distributions analogous to those for the circular and semi-circular laws. Already for $d=2$, numerical experiments reveal interesting behavior. In the following, $A_0$ and $A_1$ are independent GUE matrices and $P(z)=z^2-A_1z-A_0$. Even though the coefficients of $P$ are Hermitian, the eigenvalues of $P$ are not all real (in fact, half will be real and the other half complex).
\begin{figure}[ht]
    \centering

    \begin{subfigure}[t]{0.3\textwidth}
        \centering
        \includegraphics[width=\linewidth]{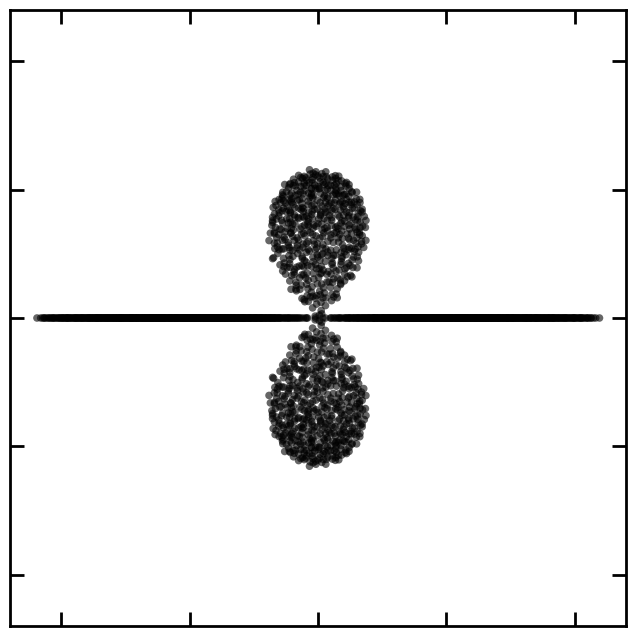}
        \caption{Scatter plot of all eigenvalues of $P$.}
        \label{a2}
    \end{subfigure}
    \hspace{2mm}
    \begin{subfigure}[t]{0.3\textwidth}
        \centering
        \raisebox{0.5\height}{%
            \centering
            \includegraphics[width=\linewidth]{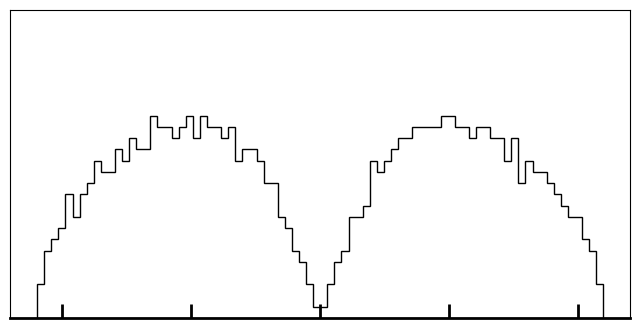}
        }
        \caption{Histogram of the real eigenvalues of $P$.}
        \label{a4}
    \end{subfigure}

    \caption{Demonstration of the Petaloid law for $P(z)=z^2-A_1z-A_0$ where $A_0$, $A_1$ are independent $1000\times1000$ GUE matrices.}
    \label{a6}
\end{figure}

Several factors make studying the PEP more difficult. Chief among them is that, unlike the SEP, the collection of eigenvalues $\lambda$ and corresponding eigenvectors $v$, i.e., vectors satisfying $P(\lambda)v=0$, do not correspond to a convenient factorization of $P$ or the $A_\ell$ matrices. Furthermore, the same eigenvector may correspond to multiple eigenvalues, and the collection of distinct eigenvectors will typically not be linearly independent (indeed, there may be $Nd$ of these in $\C^N$). For this reason, a primary approach to the PEP is a reduction to a SEP. In particular, given an $N\times N$ polynomial matrix $P$ of degree $d$, its \textit{companion linearization} is the $Nd\times Nd$ matrix, written in block form as
\[L_P=
\begin{bmatrix}
A_{d-1} & A_{d-2} & \cdots & A_1 & A_0\\
I & 0 & \cdots & 0 & 0\\
0 & I & \ddots & 0 & 0\\
\vdots & \vdots & \ddots & 0 & 0\\
0 & 0 & \cdots & I & 0
\end{bmatrix}\]
The generalized eigenvalues of $P$ are precisely the standard eigenvalues of $L_P$. This is particularly helpful in the algorithmic setting where one wishes to numerically compute the generalized eigenvalues \cite{b2,b3}.
In the random-matrix-theory setting, this linearization is not particularly suitable. In particular, natural distributions from which to sample $P$ do not push-forward to convenient distributions of $L_P$. For example, in the example depicted in \Cref{a6} where $A_0$ and $A_1$ are independent GUE matrices, $L_P$ is not a classically studied matrix distribution---it is not Hermitian and the entries above the diagonal are not independent. Worse yet, if one thinks about the entries of $P$ as random polynomials, dependence between the coefficients becomes quite natural, and then even the distinct blocks in the leading row of $L_P$ need not be independent.

This work proposes a different kind of reduction of the PEP to an SEP. Writing $P$ as
\eq{\label{a7}P(z)=z^dI-Q(z),\qwhere Q(z)=z^{d-1}A_{d-1}+\cdots+zA_1+A_0,}
see that $z\in\C$ is an eigenvalue of $P$ precisely when $z^d$ is an eigenvalue of $Q(z)$.
For finite $N$, this means the eigenvalues of $P$ are determined exactly by the eigenvalues of the collection of matrices $\set{Q(z)}_{z\in\C}$.
This work makes the following insight:
in the limit of $N\to\infty$, if, for each $z\in\C$, the ESDs of $Q(z)$ converge (in a stronger sense than weak convergence) to a known limit, then the limiting ESD of $P$ can be determined.

This insight, along with the previous work on the elliptic law of Nguyen and O'Rourke \cite{b5} which governs a class of matrices for which ESD convergence is known, results in \Cref{a8}. This theorem provides an expression for the large-$N$ limit of $\mu_P$ for a general class of distributions over polynomial matrices. The precise class is stated in \Cref{a9}; it principally includes the subclass where each matrix coefficient $A_\ell$ is a fixed complex linear combination of any number of underlying GUE matrices. More precisely, let $B_1,B_2,\ldots$ be a collection of independent GUE matrices. Fix any complex vectors $v_\ell=(v_{\ell,1},v_{\ell,2},\cdots)\in\ell_2(\C)$ and set $A_\ell=\sum_{j=1}^\infty v_{\ell,j}B_j$. For such $P$, \Cref{a8} derives a formula for the limit of $\mu_P$, and \Cref{a10} puts it into a more usable form. Furthermore, for this sub-class, using a result of Han \cite{b6}, we determine in \Cref{a11} that there are \textit{no outliers}. That is, there are almost surely no eigenvalues outside of any neighborhood of the limiting support of $\mu_P$.

Our last theorem, \Cref{a12}, concerns the double limit of large-$N$ and large-$d$. We determine a general condition in which the double limiting ESD of $\mu_P$ is the uniform distribution on the unit circle (that is, the \textit{boundary} of the unit disk).

We conclude the paper with several example distributions over $P$ and the corresponding numerically computed ESDs.

\newcommand{\cir}{\text{c}}
\renewcommand{\sc}{\text{sc}}

\subsection{Related work}
When the entries of a matrix have constant variance (as opposed to variance $O(1/N)$ as we assumed above), the rough order of size of the eigenvalues scales as $O(\sqrt N)$, and so some normalization is needed. In the SEP, this normalization can be done either by scaling the $N\times N$ matrix down by a factor of $1/\sqrt N$, or contracting the ESD inward by a factor of $1/\sqrt N$. In the SEP case, these operations are identical since
the roots of $\det(z\sqrt N I-A)$ and $\det(zI-\frac1{\sqrt N}A)$ coincide. For the monic PEP, these differ: if $P(z)=z^dI-Q(z)$, then the eigenvalues of $\det P(z\sqrt N)=\det((z\sqrt N)^dI-Q(z\sqrt N))$ are not the same as the eigenvalues of $\det(z^dI - \frac1{\sqrt N}Q(z))$.

The kind of normalization where the entries in the coefficients of $P$ have unit variance and one wants the distribution of the roots of $\det P(z\sqrt N)$ as $N\to\infty$ was recently studied by Barbarino and Noferini \cite{b7}. They show when the coefficient matrices $A_j$ are independent, unnormalized $N\times N$ complex Ginibre matrices, the limiting distribution is $\frac{d-1}d\delta_0+\frac1d\mu_\cir$ where $\mu_\cir$ is the uniform distribution on the disk of radius one, and $\delta_0$ is the Dirac delta at the origin. Their argument is to treat $P(z)$ as a perturbation to the polynomial matrix $z^dI-A_{d-1}z^{d-1}$, whose eigenvalues are simply the eigenvalues of $A_{d-1}$ concatenated with $N(d-1)$ copies of $0$.
For this reason, we speculate that, in wider generality, that there is convergence of the uniform distribution of the roots of $\det P(z\sqrt N)$ to $\frac{d-1}d\delta_0+\frac1d\lim\limits_{N\to\infty}\mu_{\frac1{\sqrt N}A_{d-1}}$. This paper focuses on the normalization where the coefficient matrices are normalized to have entries with variance $O(1/N)$.

A related rich area of research investigates the distribution of the roots of random polynomials, where the coefficients are either picked independently or subject to a specified covariance structure \cite{b8}; this reduces to the PEP when $N=1$. In that line of work, one typically is interested in the limit as the degree $d$ tends to infinity, in contrast with this work which focuses on fixed $d$ and $N$ tending to infinity.
One broad take-away from the random polynomial literature is the universality of concentration of roots around the unit circle \cite{b9,b10,b11,b12}. We observe a similar phenomenon: $\lim_{d\to\infty}\lim_{N\to\infty}\mu_P$ is the uniform measure on the circle.

\subsection{Notation and terminology}
When we use the terminology ``the eigenvalues of $P(z)$'', we always are referring to the $N$ standard eigenvalues of the matrix which is $P$ evaluated at $z\in\C$. When we want the $Nd$ generalized eigenvalues of the polynomial matrix, we will say the eigenvalues of $P$ or $P(\cdot)$. The letters $P$, $Q$, and $d$ will always be what they are in \Cref{a7}.

Similarly, $f_N(\cdot)\to f(\cdot)$ or just $f_N\to f$ means convergence of the sequence of functions $f_N$ to $f$ in $\Lloc1$ as $N\to\infty$, whereas $f_N(z)\to f(z)$ always means convergence of the sequence of scalars $f_N(z)$ to the scalar $f(z)$ for a particular $z\in\C$. A function $f$ may also be written as $(z\mapsto f(z))$.
For measures, $\mu_N\to\mu$ means that $\mu_N$ weakly converges to $\mu$ as $N\to\infty$.
$\delta$ is the Dirac delta function.
For a polynomial matrix $P\in(\C[\cdot])^{N\times N}$ with generalized eigenvalues $\lambda_1,\ldots,\lambda_{Nd}$, its ESD is
\(\mu_P(\mathcal B)=\frac1{Nd}\#\set{j\in[Nd]:\lambda_j\in\mathcal B}\), which is not to be confused with $\mu_{P(z)}$.
For a matrix $X\in\C^{N\times N}$ with eigenvalues $\lambda_1,\ldots,\lambda_{N}$, its ESD is
\(\mu_X(\mathcal B)=\frac1{N}\#\set{j\in[N]:\lambda_j\in\mathcal B}\). $X$ and $P$ often implicitly stand for a sequence of matrices and polynomial matrices of each size $N$, and the limits of $\mu_P$ and $\mu_X$ refer to the limits as $N\to\infty$.
$\H^d$ is the $d$-dimensional Hausdorff measure, scaled to agree with the Lebesgue measure.

\newcommand{\muel}{\mu^{\textnormal{ell}}}
\newcommand{\nuel}{\nu^{\textnormal{ell}}}
\section{The Petaloid Law}
As outlined in the introduction, the central idea in our proof is that the polynomial eigenvalue problem can be interpreted as a collection of standard eigenvalue problems of the same size: from $P(z)=z^d-Q(z)$, one sees that $z$ is an eigenvalue of $P$ if and only if $z^d$ is an eigenvalue of $Q(z)$.
This suggests that understanding the spectral distributions of $Q(z)$ for each fixed $z\in\C$ may lead to understanding the spectrum of $P(\cdot)$.

However, it is important to note that knowing the weak large-$N$ limits of $\mu_{Q(z)}$ for each $z$ is not enough to determine \textit{anything} about the spectrum of $P$.
Namely, the situation where $Q(z)$ either always or never contains a single eigenvalue exactly at $z^d$ for each $z$ is perfectly consistent with any weak limit of $\mu_{Q(z)}$ as the presence or absence of a single eigenvalue does not affect the limiting spectral measure.
Nevertheless, one might attempt a heuristic argument.
Consider a tiny disk $\B$ around some point $z'$. Then the number of solutions to $z^d\in\Lambda(Q(z))$ for $z$ inside $\B$, one may conjecture, is similar to the number of solutions to $z^d\in\Lambda(Q(z'))$.
This would make $\mu_P$, locally near $z'$, be the pull-back of $\mu_{Q(z')}$ under the map $z\mapsto z^d$.

The formula that this heuristic implies for the limiting spectral distribution turns out not to match empirical evidence. Our repair is as follows: instead of counting exact solutions to $z^d\in\Lambda(Q(z))$ in a small disk, we look at the geometric mean of the distances between $z^d$ and points in $\Lambda(Q(z))$. This is exactly a power of the characteristic polynomial of $Q(z)$ evaluated at $z^d$, namely,
\[\text{mean distance}=\abs{\det(z^dI-Q(z))}^{1/N}=\abs{\det(P(z))}^{1/N}.\] On the other hand, $\det(P(z))$ is monic and its roots are exactly the eigenvalues of $P$. Thus, $\abs{\det(P(z))}^{1/(Nd)}$ is the geometric mean of the distances between $z$ and the eigenvalues of $P$.
The ESDs can be recovered from these geometric means by taking the logarithm and then the distributional Laplacian, which establishes the formal relationship between the spectra of $Q(z)$ and $P$ which survives a large-$N$ limit.

The proof of \Cref{a8} is in three parts. First, in \Cref{a13}, we characterize the form of matrices with known limiting spectral densities we can take $Q(z)$ to have. Next, in \Cref{a14}, we introduce the logarithmic potential theory required (equivalent to the aforementioned geometric means). Lastly, in \Cref{a15}, we combine everything into our main result.

\subsection{Elliptic matrices}\label{a13}
A class of matrices for which the limiting standard ESD is known are called the \textit{elliptic} matrices, defined below (cf. \cite[Definition 1.3]{b5}).

\begin{definition}[Elliptic matrices]\label{a16}
Let $(\xi_1,\xi_2)$ be a centered random vector in $\C^2$ with $\E\abs{\xi_1}^2=\E\abs{\xi_2}^2<\infty$. Let $\set{x_{ij}}_{i,j\in\N}$ be an infinite double array. For each $N\in\N$, define the matrix $X\at N\in\C^{N\times N}$ with entries $X\at N_{ij}=\frac1{\sqrt N}x_{ij}$. We say that $X\at N$ is elliptic if
\begin{enumerate}
    \item The collection $\set{x_{ii}:i\in\N}\cup\set{(x_{ij},x_{ji}):i,j\in\N,i<j}$ consists of mutually independent elements.
    \item $x_{ii}$ for $1\le i$ are copies of a random variable with mean 0 and finite variance and
    \item $(x_{ij},x_{ji})$ for $1\le i<j$ are copies of $(\xi_1,\xi_2)$.
\end{enumerate}
We will say it is elliptic with correlation $\rho$, $\abs\rho\le1$ and variance $\sigma^2\ge0$ when $\E\abs{\xi_1}^2=\E\abs{\xi_2}^2=\sigma^2$ and $\rho=\E(\xi_1\xi_2)/\sigma^2$. If $\sigma^2=0$, this correlation is ill-defined, but as one may track throughout the paper, the dependence on the value of $\rho$ vanishes when $\sigma^2=0$.
For concision, we often omit the superscript and just write $X=X\at N$. We emphasize that $\rho$ is not the correlation between the coordinates of the vector $(\xi_1,\xi_2)$, but rather between the coordinates of $(\xi_1,\conj{\xi_2})$, i.e.,
\[
\E\pmat{\xi_1 \\ \bar{\xi_2}}
\pmat{\xi_1 \\ \bar{\xi_2}}^*=\sigma^2\pmat{ 1 & \rho \\ \bar \rho & 1 }.\]
\end{definition}
\begin{remark}
    When the coordinates of $(\xi_1,\xi_2)$ are i.i.d., these matrices include Ginibre ensembles and are governed by the circular law, corresponding to the $\rho=0$ case. When $(\xi_1,\xi_2)$ is supported on the set $\set{(z,\bar z):z\in\C}$ and the diagonal of $X$ is real, these matrices include GUE and GOE matrices and are governed by the semi-circular law, corresponding to $\rho=1$.
\end{remark}

When $(\xi_1,\xi_2)$ is jointly Gaussian, it is known that $\mu_X$ almost surely weakly converges to the uniform distribution on an ellipse, whose parameters depend on the variance $\sigma^2$ and correlation $\rho$ \cite{b13}. Denote the filled-in ellipse with foci $\pm f$ and semi-major axis $a$,
\[\calE_{f,a}=\set{z\in\C:\abs{z-f}+\abs{z+f}\le2a}.\]
For a matrix with variance 1 and correlation $\rho$, the relevant foci are $\pm\sqrt{4\rho}$ and the semi-major axis is $1+\abs\rho$.
When $\rho$ is real, the corresponding ellipse can be conveniently characterized explicitly in terms of Cartesian coordinates,
\[
\calE_{\sqrt{4r},1+\abs r}=\set{z\in\C:\frac{\Re(z)^2}{(1+r)^2} + \frac{\Im(z)^2}{(1-r)^2} \le 1},\, r\in(-1,1),\quad\calE_{2,2}=\conv\set{-2,2},\quad\calE_{2i,2}=\conv\set{2i,-2i}.
\]
Notice in particular that the major and minor semi-axes are $1+\abs r$ and $1-\abs r$ respectively. The picture for complex $\rho$ is simply rotated by $\frac12\arg\rho$.
If $X$ has Gaussian entries and is elliptic with variance 1 and complex correlation $\rho$, $\abs\rho<1$, the elliptic law states that
\[
\mu_X\to\frac1{\pi(1-\abs\rho^2)}\cdot\H^2\big|_{\calE_{\sqrt{4\rho}, 1+\abs\rho}},\]
where the right-hand side is the uniform distribution on the non-degenerate ellipse $\calE_{\sqrt{4\rho}, 1+\abs\rho}$. 
Taking the limit as $\abs\rho\to1$, one obtains a (rotation of the) semi-circular law; when $\abs\rho=1$, the ellipse becomes the line segment connecting the foci $\sqrt{4\rho}=\pm2e^{\frac i2\arg\rho}$, which is simply the real interval $[-2,2]$ when $\rho=1$. For non-unit variance $\sigma^2>0$, these measures can simply be dilated by $\sigma$.
For these limiting measures of $\mu_X$, we introduce abbreviated notation,
\eq{\label{a17}\muel_{\rho,\sigma}=\frac1{\pi\sigma^2(1-\abs\rho^2)}\cdot\H^2\big|_{\sigma\calE_{\sqrt{4\rho},1+\abs\rho}}\text{ for }\abs\rho<1,\quad \muel_{\rho,\sigma}=\lim_{r\to\rho}\muel_{r,\sigma}\text{ for }\abs\rho=1,\quad \muel_{\rho,0}=\delta_0.}
In the non-Gaussian case, full universality has not yet been determined.
To the authors' knowledge, the most general theorem in the literature is due to Nguyen and O'Rourke, who determine the following: say $X$ is elliptic with variance 1 and correlation $\rho\in(-1, 1)$. Under the assumption that there exists $\mu\in[0,1]$ such that
\eq{\label{a18}\E ZZ^\top=\pmat{ \mu & 0 & \mu\rho & 0 \\ 0 & 1-\mu & 0 & -(1-\mu)\rho \\ \mu\rho & 0 & \mu & 0 \\ 0 & -(1-\mu)\rho & 0 & 1-\mu }\qwhere
Z=\pmat{\Re(\xi_1)&\Im(\xi_1)&\Re(\xi_2)&\Im(\xi_2))}^\top,}
then $\mu_X\to\muel_{\rho,1}$.

By replacing $X$ with $\sigma e^{\frac i2\arg\rho}X$, this theorem immediately generalizes to all $\sigma$ and complex $\rho$ with $\abs\rho<1$. The direct analog of the condition \Cref{a18} (which is \cite[Definition 1.6 / Remark 1.7]{b5}) for handling non-real $\rho$ and non-unit $\sigma$ is the following.

\begin{definition}[C1]\label{a19}
$X$ satisfies \Cone if the following two conditions are met. First, it is elliptic, say with variance $\sigma^2$, correlation $\rho$, and underlying random elements $(\xi_1,\xi_2)$. Second, there exists $\tau$ with $\abs\tau\le1$ and $\tau\bar\rho\in\R$ such that
\[\E ZZ^* = \sigma^2\pmat{ 1 & \rho & \tau & \tau\bar\rho \\ \bar\rho & 1 & \tau\bar\rho & \bar\tau \\ \bar\tau & \tau\bar\rho & 1 & \bar\rho \\ \tau\bar\rho & \tau & \rho & 1 }\qwhere Z=\pmat{ \xi_1 & \bar{\xi_2} & \bar{\xi_1} & \xi_2 }^\top.\]
\end{definition}
\begin{remark}Both GUE and complex Ginibre matrices satisfy \Cone with $\tau=0$.\end{remark}

Then, the elliptic law of Nguyen and O'Rourke can be stated more generally as the following.
\begin{theorem}[{\cite[Theorem 1.8]{b5}}]\label{a20}
If $X$ satisfies \Cone with variance $\sigma^2>0$ and correlation $\rho$ with $\abs\rho<1$, then $\mu_X\to\muel_{\rho,\sigma}$ almost surely.
\end{theorem}

\subsection{Logarithmic potential}\label{a14}
The logarithmic potential of a measure $\mu$ is the function $U^\mu:\C\to\R\cup\set{\pm\infty}$,
$$U^\mu(z):=-\int\log\abs{w-z}\wrt\mu(w).$$
Since $\log\abs{w-\cdot}$ is the Newtonian potential on the plane,
its distributional Laplacian $\Delta=4\partial_z\partial_{\bar z}$ recovers the measure,
\(-\frac1{2\pi}\Delta U^\mu(\cdot)=\mu(\cdot).\)
Furthermore, pointwise convergence of $U^\mu$ almost everywhere implies convergence of the corresponding measures in the vague topology. A proof of this fact is sketched by Tao \cite[Theorem 2.8.3]{b14} and is written fully here as \Cref{a21,a22}. In these lemmas, $\magn\cdot$ denotes the $L^2$ norm.

\begin{lemma}\label{a21}
For each compact set $K$, there exists a constant $C_K$ such that for any probability measure $\mu$, $$\magn{U^\mu\cdot\mathbf1_K}\le C_K\pare{1+\int_\C\log(1+\abs w)\wrt\mu(w)}.$$\end{lemma}
\begin{proof}
Set $K'=K+\disk(0,1)$. $\vol(\cdot)$ denotes the 2-dimensional Lebesgue measure on $\C$.
Decompose $U^\mu(z)=U_1(z)+U_2(z)$ with
\[U_1(z)=-\int_{K'}\log\abs{w-z}\wrt\mu(w),\quad
U_2(z)=-\int_{\C\backslash K'}\log\abs{w-z}\wrt\mu(w).\]
By the triangle inequality, $\magn{U^\mu\cdot\mathbf1_K}\le\magn{U_1\cdot\mathbf1_K}+\magn{U_2\cdot\mathbf1_K}$.
Apply Minkowski's inequality to $\pare{\int U_1(z)^2\wrt z}^{\frac12}$ to obtain
\[
\magn{U_1\cdot\mathbf1_K}
\le \int_{K'}\pare{\int_K\log^2\abs{w-z}\wrt z}^{1/2}\wrt\mu(w)
\le\mu(K')\sup_{w\in K'}\pare{\int_K\log^2\abs{w-z}\wrt z}^{1/2}
\le C_K\]
for some constant $C_K$ depending on $K$.
\spliteq{}{
\magn{U_2\cdot\mathbf1_K}
\le\vol(K)^{\frac12}\sup_{z\in K} \abs{U_2(z)}
  &=\vol(K)^{\frac12}\sup_{z\in K}\int_{\C\backslash K'}\log\abs{w-z}\wrt\mu(w)
\\&\le\vol(K)^{\frac12}\sup_{z\in K}\int_{\C\backslash K'}\log(\abs{w}+\abs z)\wrt\mu(w)
\\&\le\vol(K)^{\frac12}\sup_{z\in K}
\int_\C\pare{\log(1+\abs{w})+\log(1+\abs z)}\wrt\mu(w)
\\&\le\vol(K)^{\frac12}\pare{ \int_\C\log(1+\abs w)\wrt\mu(w) + \sup_{z\in K}\log(1+\abs z) }.}
Finally, observe that the second term in the integrand does not depend on $\mu$.
\end{proof}
\begin{lemma}\label{a22}
Let $\set{\mu_N}_{N\in\N}$ be a sequence of random probability measures and $f:\C\to\R$ a measurable function. Assume that
$\limsup_{N\to\infty}\int\log(1+\abs w)\mu_N(w)<\infty$ almost surely,
and that there is a deterministic set $S\subset\C$ of full Lebesgue measure such that $z\in S$ implies $U^{\mu_N}(z)\to f(z)$ almost surely. Then $\mu_N\to-\frac1{2\pi}\Delta f$ almost surely.
\end{lemma}
\begin{proof}
Set $f_N=U^{\mu_N}$. We condition the sequence $f_N$ on the intersection of two probability 1 events. The first event is that there is an $N_0$ large enough so that $\sup_{N\ge N_0}\int\log(1+\abs w)\wrt\mu_N(w)<\infty$, which is indeed probability 1 by the lemma hypothesis.
The second event is that $f_N(z)\to f(z)$ for almost all $z$. The hypothesis doesn't immediately state that this is probability 1; the hypothesis instead says $\Pr(f_N(z)\to f(z))=1$ for $z\in S$ where $S$ has full measure (in particular, the quantifier over $z$ is outside the probability). Nevertheless, observe that this implies
\[
\int_\C\Pr(f_N(z)\not\to f(z))\wrt z
=
\int_S\Pr(f_N(z)\not\to f(z))\wrt z
+
\int_{\C\backslash S}\Pr(f_N(z)\not\to f(z))\wrt z=0+0.\]
By Tonelli's theorem, this integral is equal to
\[\E\H^2(\set{z\in\C:f_N(z)\not\to f(z)})=0.\]
Since the area is nonnegative, this implies
\[\Pr\pare{f_N(z)\to f(z)\text{ for almost all }z } = 1.\]
as desired.

Since $\int\log(1+\abs w)\wrt\mu_N(w)$ is eventually uniformly bounded, the tail sequence $\mu_N$ is tight, so convergence of $\mu_N$ to $-\frac1{2\pi}\Delta f$ in the vague topology suffices. Since $\mu_N=-\frac1{2\pi}\Delta U^{\mu_N}$, $\mu_N$ converges to $-\frac1{2\pi}\Delta f$ in the vague topology if $U^{\mu_N}$ converges to $f$ in $\Lloc1$.
It suffices to argue for
any compact set $K\subset\C$ that
\[\int_K\abs{f_N(z)-f(z)}\wrt z\to0.\]
For any $M>0$,
\spliteq{}{
\int_K\abs{f_N(z)-f(z)}\wrt z
&\le
\int_K\pare{\abs{f_N(z)-f(z)}\wedge M}\wrt z
+\frac1M\int_K\abs{f_N(z)-f(z)}^2\wrt z
\\&\le
\int_K\pare{\abs{f_N(z)-f(z)}\wedge M}\wrt z
+\frac{2\magn{ f_N\cdot\mathbf1_K }^2+2\magn{ f\cdot\mathbf1_K }^2}M.}
The dominated convergence theorem implies that the first term converges to 0.  $\magn{f_N\cdot\mathbf1_K}$ is uniformly bounded by \Cref{a21}, which in turn implies $\magn{f\cdot\mathbf1_K}$ is uniformly bounded by Fatou's lemma. Taking the limit as $M\to\infty$ gives the result.
\end{proof}

The logarithmic potential is particularly convenient to work with when $\mu$ is the ESD of either a matrix or polynomial matrix. In particular, notice that $\det P(\cdot)$ is the monic polynomial with zeros at the $Nd$ eigenvalues of $P$, but on the other hand $\det P(z)=\det(z^d-Q(z))$ is the product of the $N$ eigenvalues of $Q(z)$ shifted by $z^d$. Namely,
\eq{\label{a23}
Nd\cdot U^{\mu_P}(z)
=-\log\abs{\det P(z)}
=N\cdot U^{\mu_{Q(z)}}(z^d).}
Our strategy is to argue pointwise convergence of $U^{\mu_P}$ via \cref{a23} borrowing central parts of the proof of the elliptic law \Cref{a20} from \cite{b5}.

The logarithmic potential of the ellipse
\(\sigma\calE_{\sqrt{4\rho},1+\abs\rho}=\calE_{\sqrt{4\sigma^2\rho},\sigma+\sigma\abs\rho}\) is
\eq{\label{a24}
U^{\rho,\sigma}(z):=U^{\muel_{\rho,\sigma}}(z)=\begin{dcases}
    \,\frac12-\frac{\sigma^2\abs z^2-\Re(\sigma^2\conj \rho z^2)}{2(\sigma^4-\sigma^4\abs\rho^2)}-\log\sigma & z\in\sigma\calE_{\sqrt{4\rho},1+\abs\rho}
    \\\,
    -\log\abs{\frac{z+\sqrt{z^2-4\sigma^2\rho}}2}-\frac12\Re\pare{\frac{z-\sqrt{z^2-4\sigma^2\rho}}{z+\sqrt{z^2-4\sigma^2\rho}}}  & z\not\in\sigma\calE_{\sqrt{4\rho},1+\abs\rho}
\end{dcases},}
with the branch chosen so that $\sqrt{z^2-4\sigma^2\rho}\approx z$ for large $\abs z$.
When $\sigma=0$, $U^{\rho,\sigma}(z)=-\log\abs z$. When $\sigma>0$ and $\abs{\rho}<1$, $U^{\rho,\sigma}$ is $C^1$ over $\R^2$ and when $\sigma>0$ and $\abs{\rho}=1$ it is $C^0$.
Pointwise convergence of $U^{\mu_P}$ comes from the following lemma.

\begin{lemma}\label{a25}
    If $X$ satisfies $\Cone$ with variance $\sigma^2$ and correlation $\rho$ for $\abs\rho<1$, then for each $z\in\C$, we have $U^{\mu_X}(z)\to U^{\rho,\sigma}(z)$, defined in \cref{a24}, almost surely.
\end{lemma}
\begin{remark}
    It is important that this holds for {\em all} $z$, and not just for {\em almost} all $z$.
\end{remark}

The proof of \Cref{a25} simply combines some lemmas of \cite{b5}. To import these lemmas, we need to introduce the \textit{Hermitized} spectral distribution (EHSD) $\nu_X$, which is the density on the real line
\[
\nu_X((-\infty,a])=\frac{\#\set{j\in[N]:\sigma_j(X)\le a}}{N}.
\]
Since the absolute determinant is the product of the singular values and also the product of the absolute eigenvalues, we have the key relationship between $\mu_X$ and $\nu_X$,
\eq{\label{a26}N\cdot U^{\mu_X}(z)=-\log\abs{\det(zI-X)}=-N\cdot\int_0^\infty\log(s)\wrt\nu_{zI-X}(s).}

\begin{proof}[Proof of \Cref{a25}]
When $\sigma^2=0$, $X$ is diagonal with entry-wise variance $O(1/N)$; in particular, the spectral radius almost surely converges to 0, so $U^{\mu_X}(z)\to-\log\abs z$ as required. Now assume $\sigma^2>0$. Notice that $\frac1\sigma e^{-\frac i2\arg\rho} X$ satisfies $\Cone$ with variance $1$ and correlation $\abs\rho$, so by replacing $X$ with $\frac1\sigma e^{-\frac i2\arg\rho}X$ it suffices to consider $\sigma=1$ and $\rho\in[0,1)$. There is a collection of measures $\set{\nuel_{z,\rho}}_{z\in\C}$ such that
\eq{\label{a27}U^{\rho,1}(z)=-\int_0^\infty\log(s)\wrt \nuel_{z,\rho}(s)}
for all $z\in\C$, see \cite{b15,b5}. Furthermore, when the entries of $X$ are Gaussian, \cite[Lemma 7.17]{b5} states for almost every $z\in\C$ that $\nu_{zI-X}\to\nuel_{z,\rho}$ almost surely; the given proof holds for any $z\in\C$, which is the version we use.
Next, \cite[Lemma 7.14]{b5} explicitly states that the limiting value of $\nu_{zI-X}$ depends only on the first two moments of $X$ for almost all $z\in\C$, and again the given proof holds for any $z\in\C$.
This shows $\nu_{zI-X}\to\nuel_{z,\rho}$ for all $z$ and any $X$ satisfying $\Cone$ with variance 1 and correlation $\rho$.

Next, the proof of 
\cite[Lemma 7.5]{b5} states that for every $z\in\C$ there is some $p>0$ such that
\[\limsup_{N\to\infty}\int_0^\infty s^p\wrt\nu_{zI-X}(s)<\infty,\qand
\limsup_{N\to\infty}\int_0^\infty s^{-p}\wrt\nu_{zI-X}(s)<\infty\]
almost surely. Since there is some constant $C_p$ with $\log^2(s)\le C_p\max(s^p,s^{-p})$, this means $$\limsup_{N\to\infty}\int_0^\infty\log^2(s)\wrt\nu_{zI-X}<\infty$$ almost surely as well. This establishes that $\log(\cdot)$ is almost surely uniformly integrable with respect to the tail sequence $\set{\nu_N}_{N>N_0}$ for sufficiently large $N_0$. Therefore, with \Cref{a26,a27},
\[\lim_{N\to\infty}U^{\mu_X}(z)=-\lim_{N\to\infty}\int_0^\infty \log(s)\wrt\nu_{zI-X}(s)=-\int_0^\infty \log(s)\wrt\nuel_{z,\rho}(s)=U^{\rho,1}(z)\]

\end{proof}

With this in place, we are now ready to tackle the Petaloid law.

\subsection{Petaloid law}\label{a15}

\begin{definition}[elliptic polynomial matrices / C1-P]
\label{a9}
$P$ is an elliptic polynomial matrix if it is of the form
\eq{\label{a7}P\at N(z)=z^d-Q\at N(z),\qwhere Q\at N(z)=\sum_{\ell=0}^{d-1}A\at N_\ell z^\ell,\quad (A\at N_\ell)_{ij}=\frac1{\sqrt N}x_{ij\ell}}
for a triple array $\set{x_{ij\ell}}_{i,j\in\N,\ell\in\set{0,\ldots,d-1}}$ of (coupled) random elements
such that $Q(z)$ is elliptic for each $z\in\C$.
Correspondingly, $P$ satisfies C1-P if $Q(z)$ satisfies C1 (\Cref{a19}) for each $z\in\C$.
\end{definition}

A straightforward sufficient condition for $P$ to be elliptic is the following, analogous to \Cref{a16}. Let $(q_1,q_2)$ be a random pair of centered degree at most $d-1$ polynomials with $\E\abs{q_1(z)}^2=\E\abs{q_2(z)}^2<\infty$ for all $z$. Let $\set{f_{ij}}_{i,j\in\N}$ be an infinite double array of polynomials. For each $N\in\N$, let $Q\at N\in\C[z]^{N\times N}$ with entries $Q_{ij}\at N=\frac1{\sqrt N}f_{ij}$. Then $P$ for $P\at N(z)=z^d-Q\at N(z)$ is indeed elliptic if
\begin{enumerate}
    \item The collection $\set{f_{ii} : i\in\N}\cup\set{(f_{ij},f_{ji}):i,j\in\N, i<j}$ consists of mutually independent elements.
    \item $f_{ii}$ for $1\le i$ are copies of a random polynomial of degree at most $d-1$ with mean 0 and finite variance and
    \item $(f_{ij},f_{ji})$ for $1\le i<j$ are copies of $(q_1,q_2)$.
\end{enumerate}
\noindent Furthermore, C1-P is satisfied if there are polynomials $\sigma^2(\cdot,\cdot)$ and $g(\cdot)$ such that
\eq{\label{a28}
\E ZZ^*=\pmat{ \sigma^2(z,\bar z) & g(z) & 0 & 0 \\ \conj{g(z)} & \sigma^2(z,\bar z) & 0 & 0 \\ 0 & 0 & \sigma^2(z,\bar z) & \conj{g(z)} \\ 0 & 0 & g(z) & \sigma^2(z,\bar z) }\qwhere Z=\pmat{ q_1(z) & \conj{q_2(z)} & \conj{q_1(z)} & q_2(z) }^\top.}

Even in the context of the full \Cref{a9}, these polynomials $\sigma^2$ and $g$ will be helpful to define. For any $i<j$,
\spliteq{\label{a29}}{
\sigma^2(z,\bar z)
&=N\E\abs{Q_{ij}(z)}^2
=N\E\abs{Q_{ji}(z)}^2
\\
g(z)
&=N\E\pare{Q_{ij}(z)Q_{ji}(z)}
\\
\rho(z,\bar z)&=\frac{g(z)}{\sigma^2(z,\bar z)}
.}

\begin{remark}\label{a30}
    \Cref{a9} is technically more general than the subsequent characterization:
    consider the case of $d=2$ where $A_0$ and $A_1$ are Ginibre matrices coupled such that $(A_0)_{12}=\overline{(A_1)_{13}}$ and $(A_1)_{12}=-\overline{(A_0)_{13}}$. Then $Q(z)_{12}$ and $Q(z)_{13}$ are i.i.d. Gaussian scalars for each $z$ despite $Q_{12}$ and $Q_{13}$ not being independent as polynomials.
    However, as is made explicit in \Cref{a31}, the additional freedom garnered by coupling entries in this way does not result in different limiting spectral behavior. Similarly, no distinct behavior is seen by considering $\tau\neq0$, i.e., nonzero off-diagonal blocks in \Cref{a28}. Our theorem will show that the limiting ESD of $P$ depends only on the values of $\sigma^2$ and $g$, which are indeed unaffected by correlations between entries of $P$ and are the least constrained when the off diagonal blocks in \Cref{a28} are 0.\end{remark}

In order to apply \Cref{a22}, we must show that the tail of $\mu_P$ has a uniformly bounded logarithmic moment. The following lemma supplies this.

\begin{lemma}[Bounded log-moment]\label{a32}
If $P$ is elliptic, then the variance of the tail of $\set{\mu_{P\at N}}_{N\in\N}$ is almost surely uniformly bounded. In particular,
\[\limsup_{N\to\infty}\int\log\pare{1+\abs w}\wrt\mu_{P\at N}<\infty\] almost surely.
\end{lemma}
\begin{proof}
Consider the $Nd\times Nd$ matrix
\[
L=
\begin{bmatrix}
A\at N_{d-1} & A\at N_{d-2} & \cdots & A\at N_1 & A\at N_0\\
I & 0 & \cdots & 0 & 0\\
0 & I & \ddots & 0 & 0\\
\vdots & \vdots & \ddots & 0 & 0\\
0 & 0 & \cdots & I & 0
\end{bmatrix}
\]
This is the companion linearization of $P$, and the standard eigenvalues of $L$ are exactly the generalized eigenvalues of $P$. By Weyl majorization and the discrete Fourier transform, for $\omega$ a primitive $d\toth$ root of unity,
\[\int_\C\abs{w}^2\wrt\mu_{P\at N}(w)\le\frac1{Nd}\magn{L}_F^2=
\frac{d-1}d+\frac1{Nd}\sum_{\ell=0}^{d-1}\magn{A\at N}_F^2=
\frac{d-1}d+\frac1{Nd^2}\sum_{\ell=0}^{d-1}\magn{Q\at N(\omega^\ell)}_F^2
.\]
Since $Q\at N(\omega^\ell)$ is elliptic for each $\ell$, each term almost surely converges by the strong law of large numbers. Since there are finitely many terms, the sum converges as well.
\end{proof}

\begin{theorem}[Petaloid law]\label{a8}
Say $P$ is degree $d$ and satisfies \textnormal{C1-P} (\Cref{a9}). Define $\sigma^2$ and $g$ by \Cref{a29}. Let
\[\calE_z=\calE_{\sqrt{4g(z)},\sqrt{\sigma^2(z,\bar z)}+\abs{g(z)}/\sqrt{\sigma^2(z,\bar z)}}.\]
Then
\[
\mu_P\to-\frac1{2\pi d}\Delta\pare{z\mapsto
\begin{dcases}
    \,\frac12-\frac{\sigma^2(z,\bar z)\abs z^{2d}-\Re(\conj{g(z)}z^{2d})}{2(\sigma^2(z,\bar z)^2-\abs{g(z)}^2)}-\frac12\log\sigma^2(z,\bar z) & z^d\in\calE_z
    \\\,
    -\log\abs{\frac{z^d+\sqrt{z^{2d}-4g(z)}}2}-\frac12\Re\pare{\frac{z^d-\sqrt{z^{2d}-4g(z)}}{z^d+\sqrt{z^{2d}-4g(z)}}} & z^d\not\in\calE_z
\end{dcases}}.\]
\end{theorem}
\begin{proof}
By \Cref{a24}, the function inside the parenthesis on the right hand side is $z\mapsto U^{\rho(z,\bar z),\sigma(z,\bar z)}(z^d)$. 
By \Cref{a22} and \Cref{a32}, it suffices to show that there is a measure 0 set $E$ such that for every $z\in\C\backslash E$ we have
\[U^{\mu_P}(z)\to\frac 1dU^{\rho(z,\bar z), \sigma(z,\bar z)}(z^d)\]
almost surely. By \Cref{a23}, $U^{\mu_P}(z)=\frac1dU^{Q(z)}(z^d)$.
$\abs{g(z)}^2$ and $\sigma^2(z,\bar z)^2$ are both polynomials in $z$ and $\bar z$, so the set $S=\set{z\in\C:\abs{g(z)}=\sigma^2(z,\bar z)}$ is either $\C$ or else has measure 0.
 
Suppose $S$ has measure 0. Set $E=S$. For any $z\in\C\backslash E$, we have $\abs{g(z)}<\sigma^2(z,\bar z)$ and so $Q(z)$ is elliptic with absolute correlation strictly less than 1. Thus, \Cref{a25} gives $U^{Q(z)}(w)\to U^{\rho(z,\bar z), \sigma(z,\bar z)}(w)$ for all $w\in\C$; in particular we may take $w=z^d$.

Suppose $S=\C$. Set $E=\set{z\in\C:g(z)=0\text{ or }\Im(z^d/\sqrt{g(z)})=0}$, which has measure 0 since $z\mapsto z^d/\sqrt{g(z)}$ and $g$ are meromorphic. Fix any $z\in\C\backslash E$. Note $Q(z)$ is elliptic with unit absolute correlation. Decompose $Q(z)=\wt Q(z)+D_z$ where $D_z$ is diagonal and the diagonal entries of $\wt Q(z)$ are 0.
This way, $e^{-\frac i2\arg g(z)}\wt Q(z)$ is a Hermitian matrix with probability 1.
Then
\spliteq{}{
    \abs{U^{\mu_{Q(z)}}(w)-U^{\mu_{\wt Q(z)}}(w)}
      &=\frac{\abs{\log\det(w-Q(z))-\log\det(w-\wt Q(z))}}N
    \\&=\frac{\abs{\log\det\pare{I - D_z(w-\wt Q(z))^{-1}}}}N
    \\&\le\log\pare{\frac1{1 - \magn{D_z}\smagn{(w-\wt Q(z))^{-1}} }}
    \\&\le\log\pare{\frac1{1 - \magn{D_z}\Im(e^{-\frac i2\arg g(z)}w)^{-1} }}
    .}
By assumption, $\Im(e^{-\frac i2\arg g(z)}z^d)=\abs{g(z)}^{\frac12}\Im(z^d/\sqrt{g(z)})\neq0$, so the inverse in the denominator is well-defined and independent of $N$ for $w=z^d$.
By the strong law of large numbers, $\magn{D_z}\to0$ almost surely so $\lim_{N\to\infty}U^{\mu_{Q(z)}}(z^d)=\lim_{N\to\infty}U^{\mu_{\wt Q(z)}}(z^d)$.

To control the limit of $U^{\mu_{\wt Q(z)}}(z^d)$ we employ the semi-circular law. \cite[Theorem~2.5]{b16} states $\mu_{\wt Q(z)}\to\muel_{\rho(z,\bar z),\sigma(z,\bar z)}$. To retain convergence after integrating $\log\abs{w-z^d}$ against these measures, we need to control both the positive and negative contributions. $z^d$ is separated from the spectrum of $\wt Q(z)$, so $\log\abs{w-z^d}$ is uniformly lower bounded for $w$ in the spectrum of $\wt Q(z)$. The positive contribution of $\log\abs{w-z^d}$ is bounded since $\log\abs\cdot$ is sub-quadratic and, by Weyl's inequality, $\int\abs w^2\wrt\mu_{\wt Q(z)}\le\frac1N\magn{Q(z)}_F^2\to\sigma^2(z,\bar z)<\infty$ almost surely. Thus $\lim_{N\to\infty}U^{\mu_{\wt Q(z)}}(z^d)=U^{\rho(z,\bar z),\sigma(z,\bar z)}(z^d)$.

\end{proof}

The theorem in its above form may be unsatisfactory as it involves computing a distributional Laplacian. This does not amount to simply computing $\partial_x^2+\partial_y^2$ since the potential need not be $C^2$, i.e., one must really compute the Laplacian in the distributional sense. In fact, $\mu_P$ may approach a measure with a zero dimensional, one dimensional, and two dimensional component. To illustrate why, consider the polynomial
\[P(z)=z^3-Q(z),\qwhere Q(z)=z^2A_2+zA_1\]
where $A_1$, $A_2$ are independent GUE matrices. Since $P(0)=0$, $0$ is deterministically an eigenvalue and in fact has multiplicity $N$ so contributes $\frac{N}{3N}=\frac13$ mass to the limiting spectral measure. We can then ``deflate'' the polynomial, setting $\wt P(z)=z^2-\wt Q(z)$, $\wt Q(z)=zA_2+A_1$, to identify the remaining eigenvalues. When $z\in\R$, $\wt Q(z)$ is a multiple of a GUE matrix and in particular has all real eigenvalues. Consider what happens when one sweeps from $z=-\infty$ to $z=\infty$ along $\R$.
For both large and small $z$, the eigenvalues of $\wt P(z)$ are all positive since $\wt Q(z)=O(z)$ is dominated by $z^2$, and when $z=0$, $\wt P(0)=-A_1$, which has about half of its eigenvalues negative by the semi-circular law (that's about $N/2$ of them). 
Since the eigenvalues of $P(z)$ are continuous in $z$, each negative eigenvalue of $P(0)$ corresponds to two eigenvalues of $P$, one positive real and one negative real. This results in about $N$ additional real eigenvalues. 
The remaining $\approx N$ eigenvalues are scattered in a two dimensional region in the complex plane.
\begin{corollary}\label{a10}
Say $P$ is degree $d$ and satisfies \textnormal{C1-P} (\Cref{a9}). Let $A_\ell$ be the matrices in \Cref{a7}. Say the off diagonal entries of $A_\ell$ are almost surely zero for $\ell<m$, and of $A_m$ are nonzero with positive probability.
Define $\sigma^2$ and $g$ by \Cref{a29} and set
\[\wt\sigma^2(z,\bar z)=\abs z^{-2m}\sigma^2(z,\bar z),\qand
\wt g(z)=z^{-2m}g(z).\]
These are polynomials. Set
\spliteq{}{
\shape_\oned
&
=\set{z\in\C:
\begin{aligned}
\abs{\wt g(z)}&=\wt\sigma^2(z,\bar z)
\\
\abs{z^{d-m}-\sqrt{4\wt g(z)}}+\abs{z^{d-m}+\sqrt{4\wt g(z)}} &=2\sqrt{4\abs{\wt g(z)}}
\end{aligned}},
\\
\shape_\twod&=\set{z\in\C:
\begin{aligned}
\abs{\wt g(z)}&<\wt\sigma^2(z,\bar z)
\\
\abs{z^{d-m}-\sqrt{4\wt g(z)}}+\abs{z^{d-m}+\sqrt{4\wt g(z)}} &\le2\wt\sigma^2(z,\bar z)^{\frac12}+2\abs{\wt g(z)}/\wt \sigma^2(z,\bar z)^{\frac12}
\end{aligned}
}.
}
Then $\mu_P\to\mu_\zerod+\mu_\oned+\mu_\twod$ almost surely where
\spliteq{}{
\mu_\zerod
&=\frac md\delta_0
\\
\mu_\oned
&=\frac{\abs z^{d-m-1}\abs{2(d-m)\wt g(z)-z\wt g'(z)}\sqrt{\abs{4\wt g(z)-z^{2d-2m}}}}{4d\pi\abs{\wt g(z)}^2}\cdot\H^1\big|_{\A_\oned}.
\\
\mu_{\twod}
&=\frac1{4d\pi}\Delta\pare{z\mapsto\dfrac{\wt \sigma^2(z,\bar z)\abs z^{2d-2m}-\Re(\overline{\wt g(z)}z^{2d-2m})}{\wt \sigma^2(z,\bar z)^2-\abs{\wt g(z)}^2}
    +\log\wt\sigma^2(z,\bar z)}\cdot\H^2\big|_{\shape_\twod},
}
$\mu_\oned$ and $\mu_\twod$ are absolutely continuous densities on their respective supports and the Laplacian in $\mu_\twod$ can be interpreted in the usual, non-distributional, sense.
\end{corollary}
\begin{proof}%
Let $\wt P(z)=z^{d-m}-\sum_{\ell=m}^{d-1}A_\ell z^{\ell-m}$ and $\hat P(z)=z^m\wt P(z)$. Notice that the $\sigma^2$ and $g$ functions associated with $P$ and $\hat P$ are the same, and so $\mu_P$ and $\mu_{\hat P}$ have the same limiting measures by \Cref{a8}.
By $\det(\hat P(z))=z^{mN}\det(\wt P(z))$, the eigenvalues of $\hat P$ are the eigenvalues of $\wt P$ with $mN$ copies of 0, i.e.
\[\mu_{\hat P}=\frac{d-m}d\mu_{\wt P} + \frac md\delta_0.\]
Next, the $\sigma^2$ and $g$ functions associated with $\wt P$ are $\wt \sigma^2$ and $\wt g$, and $\wt\sigma^2(0,0)\neq0$. Write $\wt\sigma(z,\bar z)=\sqrt{\wt\sigma^2(z,\bar z)}$. \Cref{a8} gives $\mu_{\wt P}\to-\frac1{2\pi}\Delta f$ for $f(z)=\frac1{d-m}U^{\wt g(z)/\wt\sigma^2(z,\bar z),\wt\sigma(z,\bar z)}(z^{d-m})$.
    By \Cref{a24}, for
\spliteq{\label{a33}}{
\psi_\inn(z,w)
&=\frac12-\frac{\wt\sigma^2(z,w)\abs z^{2d-2m}-\Re(\conj{\wt g(z)} z^{2d-2m})}{2(\wt\sigma^2(z,w)^2-\abs{\wt g(z)}^2)}-\frac12\log\wt\sigma^2(z,w),
\\
\psi_\out(z)
&=-\log\pare{\frac{z^{d-m}+\sqrt{z^{2d-2m}-4\wt g(z)}}2}-\dfrac12\pare{\frac{z^{d-m}-\sqrt{z^{2d-2m}-4\wt g(z)}}{z^{d-m}+\sqrt{z^{2d-2m}-4\wt g(z)}}},
}
we have
\spliteq{}{
f(z)
=\frac1{d-m}
\begin{dcases}
    \psi_\inn(z,\bar z) &
    z^{d-m}\in\calE_{\sqrt{4\wt g(z)},\wt\sigma(z,\bar z)+\abs{\wt g(z)}/\wt\sigma(z,\bar z)}
    \\\,
    \Re(\psi_\out(z)) & 
    z^{d-m}\not\in\calE_{\sqrt{4\wt g(z)},\wt\sigma(z,\bar z)+\abs{\wt g(z)}/\wt\sigma(z,\bar z)}
\end{dcases}.
}
By continuity of $U^{\rho,\sigma}$, the two pieces agree, $\psi_\inn(z,\bar z)=\Re(\psi_\out(z))$, on the boundary of the ellipse. If $\abs\rho=1$ then $\calE_{\sqrt{4\rho},1+\abs\rho}$ is a line segment meaning $f(z)=\frac1{d-m}\Re(\psi_\out(z))$ for each $z$ with $\abs{g(z)}=\sigma^2(z,\bar z)$. Next, notice
\spliteq{}{
\set{z\in\C:\abs{\wt\rho(z,\bar z)}<1}\cap
\calE_{\sqrt{4\wt g(z)},\wt\sigma(z,\bar z)+\abs{\wt g(z)}/\wt\sigma(z,\bar z)}
=\A_\twod.}
Thus, we may write $f$ as
\[f(z)=\frac1{d-m}\begin{dcases}\psi_\inn(z,\bar z) & z\in\A_\twod\\\Re(\psi_\out(z)) & z\not\in\A_\twod\end{dcases}.\]
This immediately implies $\Delta f\big|_{\A_\twod}=\frac1d\Delta \psi_\inn(z,\bar z)\big|_{\A_\twod}$.
On $\A_\twod$, notice that $\psi_\inn(z,\bar z)$ is smooth and so the Laplacian can be computed directly.
If $g=0$ then $\A_\oned=\emptyset$ so the theorem would be concluded. Now assume $g\neq 0$.
Our cut of $\psi_\out$ ensures that it is locally holomorphic when $z^{d-m}$ is not in the line segment connecting the two roots of $4\wt g(z)$, i.e., the condition
\eq{\label{a34}\abs{z^{d-m}-\sqrt{4\wt g(z)}}+\abs{z^{d-m}+\sqrt{4\wt g(z)}}\neq2\sqrt{4\abs{\wt g(z)}}}
implies $\psi_\out$ is locally holomorphic at $z$.
Set $h(z)=z^{2d-2m}/\wt g(z)$ and notice \Cref{a34} is met exactly when $z\not\in h^{-1}([0,4])$, and $A_\oned=h^{-1}([0,4])\backslash\A_\twod$.
By the Cauchy-Riemann equations, this means $\Delta f\big|_{\C\backslash(\A_\oned\cup\A_\twod)}=0$.
For the measure on $\A_\oned$, first consider the logarithmic potential of the semi-circular distribution, $U^{1,1}$. This is a one dimensional measure so the chain-rule gives
\[
-\frac1{2\pi}\Delta U^{1,1}=\frac1{2\pi}\sqrt{4-z^2}\cdot\H^1\big|_{[-2,2]}
\implies
-\frac1{2\pi}\Delta(U^{1,1}\circ p)=\frac1{2\pi}\sqrt{4-p(z)^2}\,\abs{p'}\cdot\H^1\big|_{p^{-1}([-2,2])}
\]
For $p=\sqrt h$, notice $U^{1,1}\circ p=\Re(\psi_\out)+\frac12\log\abs{\wt g}$ and $p^{-1}([-2,2])=h^{-1}([0,4])$. Since the semi-circular law has an even density function, the branch of $\sqrt h$ is immaterial. On $\A_\oned$, observe that $\wt g$ cannot have a zero. If it did, membership in $\A_\oned$ implies $\abs z=0$, but then we would have $\abs{\wt g(0)}=\wt\sigma^2(0,0)=0$ which is ruled out by construction of $\wt\sigma$.
This means $\log\abs{\wt g(z)}$ is harmonic on $\A_\oned$ so vanishes under application of the Laplacian. Therefore
\spliteq{}{
-\Delta(U^{1,1}\circ p)
=
-\Delta\Re(\psi_\out(z))
  &=\sqrt{4-h(z)}\,\sabs{p'(z)}\cdot\H^1\big|_{\A_\oned}
\\&=\sqrt{4-h(z)}\,\abs{\frac{h'(z)}{2h(z)^{\frac12}}}\cdot\H^1\big|_{\A_\oned}.
}
Observe that $4-h(z)$ is a nonnegative real number for $z\in\A_\oned$, so we may replace it with $\abs{4-h(z)}$. Applying the quotient rule to $h$ and simplifying finishes the theorem.
    
\end{proof}

{
}

\subsection{Independent Wigner coefficients}
One interesting special case is when the coefficients $A_j$ are independent and elliptic with correlation 1, e.g. GUE. In this case, we can work out an explicit formula for $\mu_\oned$ for any $d$. In principle, one can do the same for $\mu_\twod$ as well, but we perform the calculations for $d=2$ only.

\begin{corollary}\label{a35}
Let $\alpha_d\in[2,\sqrt5)$ be the unique root in that interval of $x^{2d+2}-5x^{2d}+4$.
If the coefficients of $P$ are independent matrices satisfying $\Cone$ (\Cref{a19}) with $\tau=0$, variance 1, and correlation 1 (e.g., are GUE), then
\[
\mu_\oned=\frac1d\cdot\frac{\abs{x}^{d-1}\pare{x^{2d+2}-(d+1)x^2+d}}{2\pi \pare{x^{2d}-1}^2}\sqrt{\frac{x^{2d+2}-5x^{2d}+4}{1-x^2}}\cdot\H^1\big|_{[-\alpha_d,\alpha_d]}
.\]
(NB: the apparent singularities at $x=\pm1$ are removable discontinuities). If $d=2$, then for $z=x+iy$, $r^2=x^2+y^2$,
\[
\A_\twod=\set{z\in\C:(1+r^2)((r^2+2)^2x^2+r^4y^2)<4y^2}\qand
\mu_\twod
=\frac1{2\pi}\pare{1+2r^2+\frac1{(1+r^2)^2}}\cdot\H^2\big|_{\A_\twod}.\]
\end{corollary}
\begin{proof}
Apply \Cref{a8}.
Observe that $m=0$ so $\wt \sigma^2(z,\bar z)=\sigma^2(z,\bar z)=\sum_{j=0}^{d-1}\abs z^{2j}$ and $\wt g(z)=g(z)=\sum_{j=0}^{d-1}z^{2j}$.
It is clear that $\A_\oned\subset\R$, so one computes it is $[-\alpha_d,\alpha_d]$ by the intermediate value theorem.
To compute $\A_\twod$, recall that
\[
\shape_\twod=\set{z\in\C:
\begin{aligned}
\abs{g(z)}&<\wt\sigma^2(z,\bar z)
\\
    z^d&\in\calE_{ \sqrt{4g(z)},\sigma(z,\bar z)+\abs{g(z)}/\sigma(z,\bar z) }.
\end{aligned}
}
\]
A convenient formula for the ellipse with $\sigma>0$ is
\[\calE_{\sqrt{4g},\sigma+\abs g/\sigma}=\set{z\in\C: \sigma^2\abs{\sigma^2 z- g\bar z}^2 \le (\sigma^4 - \abs g^2)^2}.\]
Direct computation gives the final expressions for $d=2$.

\end{proof}

\subsection{The large-\texorpdfstring{$d$}{d} limit}
In \Cref{a36} the limiting ESD visually approaches the uniform distribution on the unit circle as $d$ gets larger and larger. In this section, we verify that this is indeed the case. In particular, we provide general conditions under which the large-$d$ limit of the large-$N$ limit of the ESD of $P$ approaches the uniform distribution on the circle. These conditions deal with the pointwise asymptotic-in-$d$ values of the $\sigma$ and $g$ functions; roughly speaking, when $\sigma(z,\bar z)$ is $\Theta(1+z^{d-O(1)})$ for each fixed $z$ they are satisfied. For instance, they are satisfied if 
the coefficients $A_j$ are independent and elliptic with any correlation and variances $\sigma_j^2$ satisfying $\sigma_0^2\neq0$ and $\sigma_j^2=\Theta(1)$.
\begin{theorem}\label{a12}
Consider a sequence of elliptic polynomial matrices $P_1,P_2,\cdots,$ satisfying \textnormal{C1-P} (\Cref{a9}) where $P_d$ has degree $d$. Denote $\mu_d=\lim\limits_{N\to\infty}\mu_{P_d}$. Let $\sigma_d^2$, $g_d$, $\rho_d$, be the functions \Cref{a29} for $P_d$. Suppose
\begin{enumerate}[1.]
    \item 
\eq{\label{a37}S:=\set{ z:\limsup_{d\to\infty}  \abs{ \rho_d(z,\bar z) } < 1\text{ and }\liminf_{d\to\infty}\sigma^2_d(z,\bar z)>0 }}
has full-measure,
\item\eq{\label{a38}\lim_{d\to\infty}\pare{\sigma^2(z,\bar z)}^{1/{2d}}=\max(1,\abs z)\quad\forall z,}
\item there exists $R$ with \eq{\label{a39} \limsup_{d\to\infty}\sup_{\abs z>R}\abs{\frac{\sigma^2(z,\bar z)}{z^{2d}}}<\frac14. }
\item
\eq{\label{a40}
\lim_{d\to\infty}\frac{z^{2d}}{d\cdot\sigma^2_d(z,\bar z)}=0\quad\forall z,}and
\item
\eq{\label{a41}
\limsup_{d\to\infty}\abs{\frac{g_d(z)}{z^{2d}}}<\infty\quad\forall\abs z>1.}
\end{enumerate}
Then $\lim\limits_{d\to\infty}\mu_d$ is the uniform distribution on the unit circle.
\end{theorem}

Let $\A\at d_\twod$ be $\A_{\twod}$ corresponding to the polynomial eigenvalue problem of degree $d$.
Our first lemma for \Cref{a12} shows that $\A_\twod$ ``fills out'' the unit disk and that the spectrum of $P$ is contained. Our next shows convergence of the logarithmic potential, which with \Cref{a22}, implies the result.

\begin{lemma} \label{a42}For every $z\in S$ with $\abs z<1$,
\[\lim_{d\to\infty}\one_{ z^d\in\calE_{\sqrt{4g_d(z)},\sigma_d(z,\bar z)+\abs{g_d(z)}/\sigma_d(z,\bar z)} }=1.\]
Moreover, for $\abs z>R$,
\[\lim_{d\to\infty}\one_{ z^d\in\calE_{\sqrt{4g_d(z)},\sigma_d(z,\bar z)+\abs{g_d(z)}/\sigma_d(z,\bar z)} }=0.\]
\end{lemma}
\begin{proof}
Put $g=g_d$ and $\sigma=\sigma_d$.
The semi-minor axis of $\calE_{\sqrt{4\rho},1+\abs\rho}$ is $1-\abs\rho$, so
\[\calE_{\sqrt{4g(z)},\sigma(z,\bar z)+\abs{g(z)}/\sigma(z,\bar z)} \supset \disk\pare{0, \sigma(z,\bar z)-\abs{g(z)}/\sigma(z,\bar z)}.\]
When $z\in S$, this is a disk of strictly positive radius in the limit as $d\to\infty$ centered at the origin. On the other hand, when $\abs z<1$, $\lim\limits_{d\to\infty}z^d=0$ and consequently
\[
\pare{ z\in S \text{ and }\abs z<1 }\implies\lim_{d\to\infty}\one_{ z^d\in\calE_{\sqrt{4g(z)},\sigma(z,\bar z)+\abs{g(z)}/\sigma(z,\bar z)} }=1
\]
On the other hand, the inclusion
\[z^d\in\calE_{\sqrt{4g(z)},\sigma(z,\bar z)+\abs{g(z)}/\sigma(z,\bar z)} \subset \disk\pare{0, \sigma(z,\bar z)+\abs{g(z)}/\sigma(z,\bar z)}
\subset\disk\pare{0, 2\sigma(z,\bar z)}\]
implies $\abs z^d\le 2\sigma(z,\bar z)$ which does not happen for $\abs z>R$ in the limit by \Cref{a39}.
\end{proof}
\begin{lemma} \label{a43}
\[\lim_{d \to \infty} U^{\mu_d}(z) = \begin{cases}
            -\log\abs z & \abs z>1 \\ 0 & \abs z<1\end{cases}\]
for each $z\in S$.
\end{lemma}
\begin{proof}
    Recall the definitions of $\psi_\inn$ and $\psi_\out$, \Cref{a33}.
    For $\abs z>1$, $U^{\mu_d}(z)$ is determined by some combination of $\psi_\inn$ and $\psi_\out$. For $\abs z<1$, by \Cref{a42}, $U^{\mu_d}(z)$ is determined solely by $\psi_\inn$.
    First consider $|z| > 1$ and $z\in S$.
    By \Cref{a41}, $\abs{g_d(z)}$ is negligible compared to $z^{2d}$ in the limit, so
    \[
        \lim_{d \to \infty} \frac{1}{d} \Re(\psi_{\out}(z))
        = -\lim_{d \to \infty} \frac{1}{d}\log |z|^d = -\log|z|.
    \]
    Now consider any $z\in S$.
    Writing $g=g_d(z)$ and $\sigma^2=\sigma_d^2(z,\bar z)$ and using \Cref{a40,a38},
    \spliteq{}{
    \lim_{d\to\infty}\frac1d\psi_\inn(z)
      &=-\lim_{d\to\infty}\frac1{2d}\log\sigma^2(z,\bar z) - \lim_{d\to\infty}\frac1{2d}\frac{\sigma^2\abs z^{2d}-\Re(\bar gz^{2d})}{(\sigma^2)^2-\abs g^2}
    \\&=-\lim_{d\to\infty}\frac1{2d}\log\sigma^2(z,\bar z) - \lim_{d\to\infty}\frac1{2d}\frac{\abs z^{2d}}{\sigma^2}\frac{1-\Re(\bar\rho z^{2d}/\abs z^{2d})}{1-\abs \rho^2}
    \\&=-\lim_{d\to\infty}\frac1{2d}\log\sigma^2(z,\bar z)
    \\&=-\log\max(1,\abs z)
    .}
\end{proof}
\begin{proof}[Proof of \Cref{a12}]
    This is a direct consequence of Gauss's shell law, and \Cref{a42,a43,a22}.
\end{proof}

\section{No Outliers}
The previous section identified regions, $\A_\oned$, $\A_\twod$, outside of which the limiting density vanishes.
In this section we show something stronger in the Gaussian case: there are almost surely no eigenvalues outside any small enlargement of $\A_\oned\cup\A_\twod$ or $\A_\oned\cup\A_\twod\cup\set0$ (depending if $m$ of \Cref{a10} is 0 or not). Our strategy is to first rule out large eigenvalues by a spectral norm bound, and then construct a net over a compact disk minus an open neighborhood of the locations where one expects some spectrum. At each point in the net, a lower bound on $\sigma_N(P(z))$ will imply $\sigma_N(P(z))>0$ at nearby points, and therefore no eigenvalues of $P$ there.

\newcommand{\F}{\mathcal F}
\begin{lemma}[{\cite[eq. (3.7)]{b6}}]\label{a44}
Let $X$ have Gaussian entries and say it satisfies $\Cone$ (\Cref{a19}) with variance $1$ and correlation $\rho\in[-1,1]$.
For each $\eps>0$, there exists $c>0$ such that\[\liminf_{N\to\infty}\inf_{z\in\C:\dist(z,\calE_{\sqrt{4\rho},1+\abs\rho})\ge\eps}\sigma_N(X_N-z)\ge c\]almost surely.
\end{lemma}

\begin{theorem}\label{a11}
    Let $\A_\oned$, $\A_\twod$, $\sigma^2$ be as in \Cref{a10}. If $\sigma(0,0)=0$, set $\A=\set0\cup\A_\oned\cup\A_\twod$ otherwise set $\A=\A_\oned\cup\A_\twod$.
    Assume $P$ has jointly Gaussian coefficients.
    For each $\eps>0$,
    \[\lim_{N\to\infty}\#\pare{\Lambda(P\at N)\cap\set{z\in\C : \dist(z,\A)>\eps}}=0\]
almost surely.
\end{theorem}
\newcommand{\G}{\mathcal G}
\begin{proof}
Let $\calE_z=\calE_{\sqrt{4g(z)},\sigma(z,\bar z)+\abs{g(z)}/\sigma(z,\bar z)}$ be the ellipse to which the spectral density of $Q(z)$ is converging.
It suffices to show for every $\alpha>0$ that
\[\Pr\pare{
\limsup_{N\to\infty}\#\pare{\Lambda(P\at N)\cap\set{z\in\C : \dist(z,\A)>\delta}}\ge1
}\le2\alpha.\]
Let $R>1$ be a large enough constant such that $\limsup_{N\to\infty}\max_{0\le k<d}\magn{A_k}\le R$ with probability $1-\alpha$. Call this event $E$. Under this event, all eigenvalues of $Q(z)$ are contained in a disk of radius
\[\sum_{k=0}^{d-1}\magn{A_k}\abs z^k\le dRz^{d-1}\]
In particular, this excludes $z^d$ when $z\ge2dR$, which means
\[E\implies \limsup_{N\to\infty}\#\pare{\Lambda(P)\cap(\C\backslash\disk(0,2dR))}=0.\]
Let $R'=2dR$.
Next, notice that $\A=\set{z\in\C:\dist(z^d,\calE_z)=0}$.
Let
\[\delta=\frac12\inf\set{ \dist(z^d,\calE_z):z\in\C,\dist(z,\A)\ge\eps }.\]
We claim that $\delta>0$. To see this, note that $\dist(z^d,\calE_z)$ is continuous in $z$, that $\dist(z^d,\calE_z)$ tends to infinity as $z$ tends to infinity, and that $\set{z\in\C:\dist(z,\A)\ge\eps}$ restricted to any compact set is compact. The infimum is therefore achieved; it cannot be achieved for $0$ since $\dist(z^d,\calE_z)=0$ would imply $\dist(z,\A)=0$. That finishes the claim.
Let\[\F=\set{z:\dist(z^d,\calE_z)\le\delta}.\]
By selection of $\delta$, if $\dist(z,\A)\ge\eps$ then $\dist(z^d,\calE_z)>\delta$ so $z\not\in\F$.
In this case, \Cref{a44} states that there exists $c_z$ such that \(\liminf_{N\to\infty}\sigma_N(z^d - Q(z))\ge c_z\) almost surely.
Then for $w,z\in\disk(0,R')$,
\spliteq{}{
E\implies \limsup_{N\to\infty}\abs{\sigma_N(w^d-Q(w)) - \sigma_N(z^d-Q(z))}
  &\le\magn{w^d-z^d-Q(w)+Q(z)}
\\&\le\abs{w^d-z^d}+R\sum_{k=0}^{d-1}\sabs{w^k-z^k}
\\&\le\abs{w-z}(R')^d
}
Set $L=(R')^d$. Let $\G$ be the closure of $\disk(0,R')\backslash\F$. Consider the open cover of $\G$,
\[\bigcup_{z\in\G}\disk(z,c_z/(2L))\supset\G.\]
Since $\G$ is compact, there exists a finite sub-cover $S$,
\[\bigcup_{z\in S}\disk(z,c_z/(2L))\supset\G.\]
Since $S$ is finite,
\[\liminf_{N\to\infty}\min_{z\in S}\frac{\sigma_N(z^d-Q(z))}{c_z}\ge1\]
almost surely. Pick $N_0$ large enough so that
\[\Pr\pare{\inf_{N\ge N_0}\min_{z\in S}\frac{\sigma_N(z^d-Q(z))}{c_z}>\frac12}\ge1-\alpha.\]
Call this event $E'$.
Now, observe that $\sigma_N(z^d-Q(z))>\frac{c_z}2$ implies $\sigma_N(w^d-Q(w))>0$ for $w\in\disk(z,c_z/(2L))$, which means this disk contains no eigenvalues of $P$. These disks cover $\disk(0,R')\backslash\mathcal F$ so
\[E\cap E'\implies \limsup_{N\to\infty}\#\pare{\Lambda(P)\cap(\disk(0,2dR)\backslash\F)}=0.\]
Since $\Pr(E\cap E')\ge1-2\alpha$ by the union bound, the desired result follows.
\end{proof}

\section{Numerical Depictions}
In this section, we provide numerical data which shows visually the densities $\mu_P$ corresponding to different elliptic polynomial matrix models. There are very many natural models one may devise. Our next Proposition shows that if one wants to capture all possible limiting spectral distributions, it suffices to consider only polynomial matrices where each matrix coefficient $A_\ell$ is a fixed linear combination of some common collection of independent GUE matrices.

\begin{proposition}\label{a31}
For each elliptic polynomial matrix $P$, there exists a matrix $C\in\C^{d\times2d}$ such that $\wt P$ defined by
\eq{\label{a45}\wt P(z)=z^d - \sum_{\ell=0}^{d-1}\sum_{j=1}^{2d}C_{\ell,j} B_jz^\ell}
where $B_1,\ldots,B_{2d}$ are independent GUE matrices
satisfies
\[\lim\limits_{N\to\infty}\mu_P=\lim\limits_{N\to\infty}\mu_{\wt P}.\]
Equivalently, there are polynomials $\set{c_j}_{j=1}^{2d}$ of degree at most $d-1$ with
\[\wt P(z)=z^d-\sum_{j=1}^{2d}c_j(z)B_j.\]
\end{proposition}
\begin{proof}
Let $\sigma^2_P$ and $g_P$ (resp. $\sigma^2_{\wt p}$ and $g_{\wt P}$) be the polynomials defined in \Cref{a29} corresponding to $P$ (resp. $\wt P$). By \Cref{a8}, it suffices to show that $\sigma^2_{\wt P}=\sigma^2_P$ and $g_{\wt P}=g_P$.
Define the matrices $\Sigma$ and $S$
by the entries
\[
\Sigma_{\ell,\ell'}
=N\E\sqbrac{\frac{(A_\ell)_{12}\conj{(A_{\ell'})_{12}}+(A_\ell)_{21}\conj{(A_{\ell'})_{21}}}2}
,\quad
S_{\ell,\ell'}=N\E\sqbrac{\frac{(A_\ell)_{12}(A_{\ell'})_{21}+(A_\ell)_{21}(A_{\ell'})_{12}}2}.\]
Then the random vectors
\spliteq{}{
v:=\sqrt{N/2}\pmat{ (A_0)_{12} & (A_1)_{12} & \cdots & (A_{d-1})_{12} & \conj{(A_0)_{21}} & \conj{(A_1)_{21}} & \cdots & \conj{(A_{d-1})_{21}} }^\top\in\C^{2d},
\\
u:=\sqrt{N/2}\pmat{ (A_0)_{21} & (A_1)_{21} & \cdots & (A_{d-1})_{21} & \conj{(A_0)_{12}} & \conj{(A_1)_{12}} & \cdots & \conj{(A_{d-1})_{12}} }^\top\in\C^{2d}
.}
have covariance structure
\eq{\label{a46}\E(vv^*+uu^*)
= \bmat{ \Sigma & S \\ S^* & \conj\Sigma }
= \bmat{ \Sigma & S \\ \conj S & \conj\Sigma }.}
Matrices of this form are called \textit{doubled-up} \cite{b17}, and they always admit a matrix $C$ such that $CC^*=\Sigma$ and $CC^\top=S$ \cite{b18}. To be explicit, one can verify that the set of such matrices is closed under application of polynomials with real coefficients, and so by the functional calculus, the square root of $\E(vv^*+uu^*)$ has the same form. That is, there are $U$ and $V$ with
\[
\E(vv^*+uu^*)=
\bmat{ U & V \\ \conj V & \conj U }
\bmat{ U & V \\ \conj V & \conj U }^*.\]
Set $C=\frac1{\sqrt2}\bmat{U+V & (U-V)i}$.
This way,
$\Sigma=CC^*$ and $S=CC^\top$. Then, one can directly compute that
\[\sigma_{\wt P}^2(z,\bar z)=\sum_{\ell,\ell'}(CC^*)_{\ell,\ell'}z^\ell\bar z^{\ell'}=\sigma_P^2(z,\bar z),\qand
g_{\wt P}(z)=\sum_{\ell,\ell'}(CC^\top)_{\ell,\ell'}z^{\ell+\ell'}=g_P(z).\]
\end{proof}

With this in place, we have a convenient way of specifying an elliptic polynomial matrix: a $d\times 2d$ complex generation matrix, $C$, corresponding to the polynomial matrix in \Cref{a45}. For many of these plots, we will specify a $C$ with fewer than $2d$ columns; these are implicitly padded with 0s. In all of the plots in this section, $d$ is the degree of $P=P\at N$ and $N=\floor{4000/d}$. Unless otherwise specified, the plotted eigenvalues are for a single sample of $P\at N$. In all plots, the window is centered at 0 and tick marks are placed at the integers.
Eigenvalues are computed via the companion linearization and NumPy's \alg{numpy.linalg.eig} function, using the reference implementation of \cite{b19}.

\begin{figure}[H]
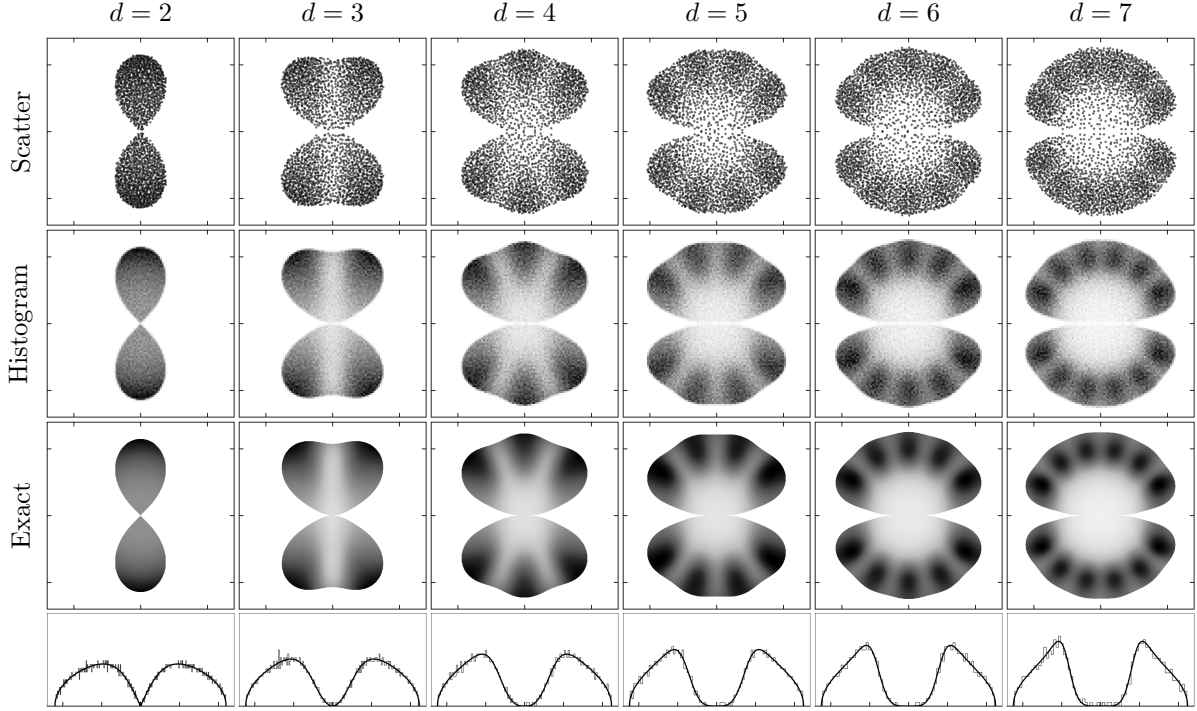

\centering
\begin{tikzpicture}[x=\linewidth/13, y=\linewidth/13, cell/.style={}, collabel/.style={align=center}, rowlabel/.style={rotate=90, align=center}]
\newlength{\gridx}
\newlength{\xgridx}
\setlength{\gridx}{\linewidth}
\divide\gridx by 13
\setlength{\xgridx}{\gridx}
\multiply\xgridx by 2
\foreach \c/\name in {0/{$d=2$}, 1/{$d=3$}, 2/{$d=4$}, 3/{$d=5$}, 4/{$d=6$}, 5/{$d=7$}}{\node[collabel] at (1+2*\c,0.25) {\name};}
\foreach \r/\name in { 0/{Scatter}, 1/{Histogram}, 2/{Exact}}{\node[rowlabel] at (-0.25,-1-2*\r) {\name};}
\foreach \di/\dval in {0/2, 1/3, 2/4, 3/5, 4/6, 5/7}{
\node[cell] (fig) at (1+2*\di,-1) {\includegraphics[width=\xgridx,keepaspectratio]{verify_scatter_d\dval.png}};
\node[cell] (fig) at (1+2*\di,-3) {\includegraphics[width=\xgridx,keepaspectratio]{verify_histogram_d\dval.png}};
\node[cell] (fig) at (1+2*\di,-5) {\includegraphics[width=\xgridx,keepaspectratio]{verify_exact_d\dval.png}};
\node[cell] (fig) at (1+2*\di,-6.5)
{\includegraphics[width=\xgridx,keepaspectratio]{verify_real_d\dval.png}};
}
\end{tikzpicture}
\caption{
$C=\bmat{I&0}$. Row 1: scatter plots of just the non-real eigenvalues. Row 2: 2d-histogram of non-real eigenvalues from 100 independent copies of $P$. Row 3: density determined by \Cref{a35} (with numerical computation of the Laplacian via finite differences).
Row 4: normalized histogram of the real eigenvalues superimposed with the density determined by \Cref{a35}.
}
\label{a36}
\end{figure}

\begin{figure}[H]
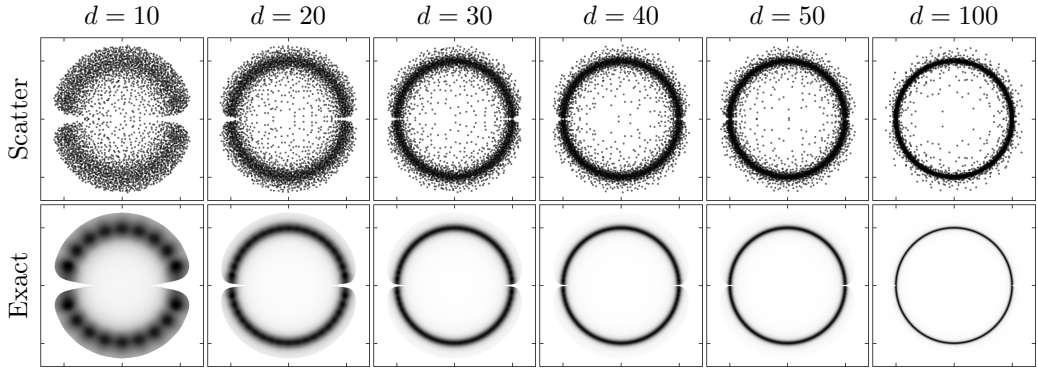

\centering
\begin{tikzpicture}[x=\linewidth/15, y=\linewidth/15, cell/.style={}, collabel/.style={align=center}, rowlabel/.style={rotate=90, align=center}]
\newlength{\gridxb}
\newlength{\xgridxb}
\setlength{\gridxb}{\linewidth}
\divide\gridxb by 15
\setlength{\xgridxb}{\gridxb}
\multiply\xgridxb by 2
\foreach \c/\name in {0/{$d=10$}, 1/{$d=20$}, 2/{$d=30$}, 3/{$d=40$}, 4/{$d=50$}, 5/{$d=100$}}{\node[collabel] at (1+2*\c,0.25) {\name};}
\foreach \r/\name in { 0/{Scatter}, 1/{Exact}}{\node[rowlabel] at (-0.25,-1-2*\r) {\name};}
\foreach \di/\dval in {0/10, 1/20, 2/30, 3/40, 4/50, 5/100}{
\node[cell] (fig) at (1+2*\di,-1) {\includegraphics[width=\xgridxb,keepaspectratio]{verify_scatter_d\dval.png}};
\node[cell] (fig) at (1+2*\di,-3) {\includegraphics[width=\xgridxb,keepaspectratio]{verify_exact_d\dval.png}};
}
\end{tikzpicture}
\caption{Comparison of empirical data with \Cref{a12}. The matrices and plots are the same as in \Cref{a36}, but with larger $d$ values.}
\label{a47}
\end{figure}

\begin{figure}[H]
\centering
\begin{tikzpicture}[x=\linewidth/16, y=\linewidth/16, cell/.style={}, collabel/.style={align=center}, rowlabel/.style={align=left, anchor=west}]
\setlength{\gridx}{\linewidth}
\divide\gridx by 16
\setlength{\xgridx}{\gridx}
\multiply\xgridx by 2
\foreach \c/\name in {0/{$d=2$}, 1/{$d=3$}, 2/{$d=4$},
3/{$d=5$},
4/{$d=6$},
5/{$d=7$},
6/{$d=8$}%
}{\node[collabel] at (1+2*\c,0.25) {\name};}
\foreach \r/\name in {
0/{$\rho=1$},
1/{$\rho=0.92$},
2/{$\rho=0.71$},
3/{$\rho=0.38$},
4/{$\rho=0$},
5/{$\rho=-0.38$},
6/{$\rho=-0.71$},
7/{$\rho=-0.92$},
8/{$\rho=-1$}
}{
\node[rowlabel] at (-1.75,-1-2*\r) {\name};
\foreach \di/\dval in {0/2, 1/3, 2/4, 3/5, 4/6, 5/7, 6/8}
    \node[cell] (fig) at (1+2*\di, -1-2*\r) {\includegraphics[width=\xgridx,keepaspectratio]{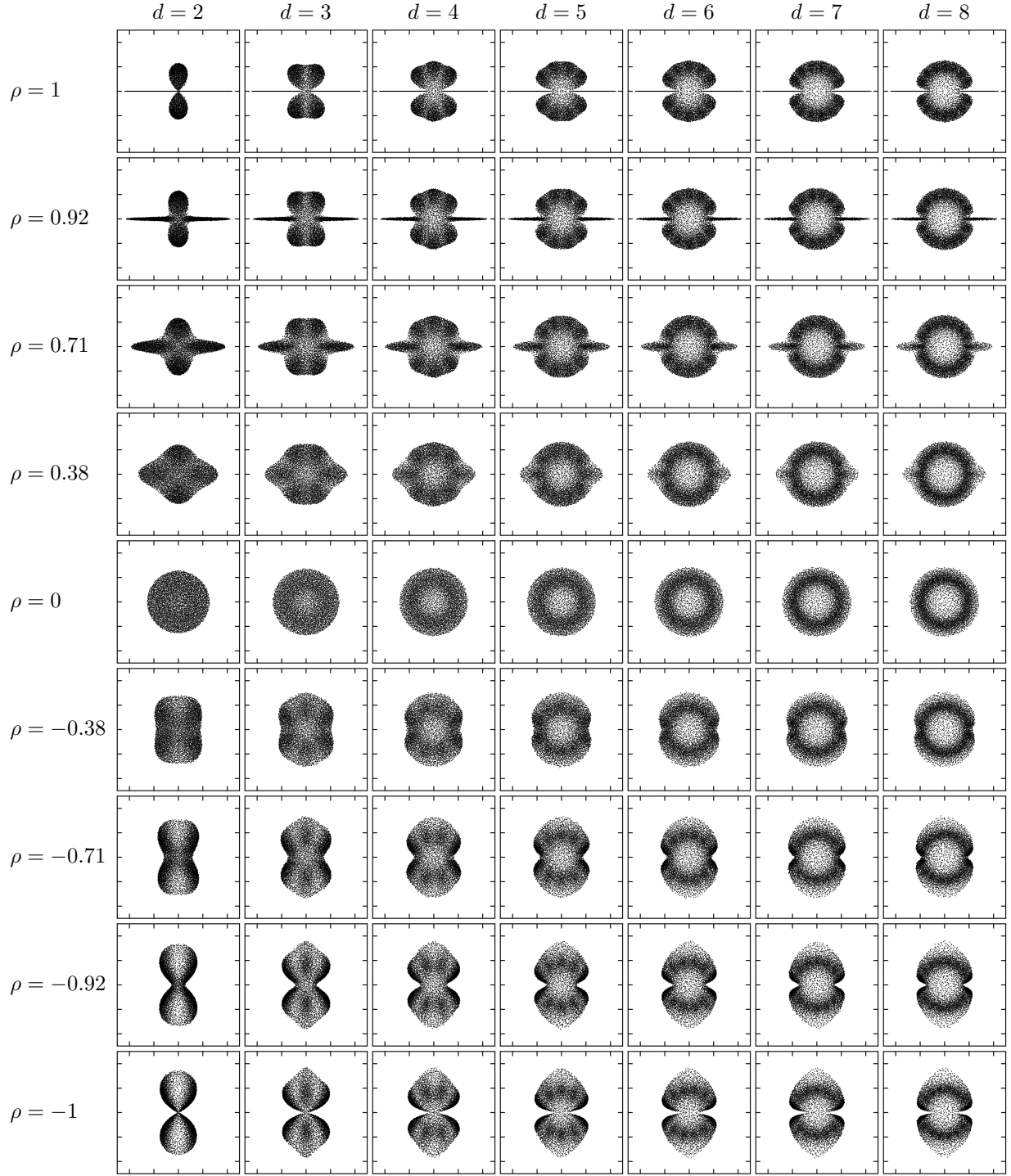}};
}
\end{tikzpicture}
\caption{ $C=I\otimes\pmat{\cos(\theta) & i\sin(\theta)}$ with $\rho=\cos(2\theta)$. This corresponds to each $A_j$ being independent and elliptic with variance 1 and correlation $\rho\in\R$. }
\label{a48}
\end{figure}

\begin{figure}[H]
\centering
\begin{tikzpicture}[x=\linewidth/16, y=\linewidth/16, cell/.style={}, collabel/.style={align=center}, rowlabel/.style={align=left, anchor=west}]
\setlength{\gridx}{\linewidth}
\divide\gridx by 16
\setlength{\xgridx}{\gridx}
\multiply\xgridx by 2
\foreach \c/\name in {0/{$d=2$}, 1/{$d=3$}, 2/{$d=4$},
3/{$d=5$},
4/{$d=6$},
5/{$d=7$},
6/{$d=8$}%
}{\node[collabel] at (1+2*\c,0.25) {\name};}
\foreach \r/\name in {
0/{$\rho=1$},
1/{$\rho=\frac{1+i}{\sqrt2}$},
2/{$\rho=i$},
3/{$\rho=\frac{-1+i}{\sqrt2}$},
4/{$\rho=-1$},
5/{$\rho=\frac{-1-i}{\sqrt2}$},
6/{$\rho=-i$},
7/{$\rho=\frac{1-i}{\sqrt2}$},
8/{$\rho=1$}
}{
\node[rowlabel] at (-1.5,-1-2*\r) {\name};
\foreach \di/\dval in {0/2, 1/3, 2/4, 3/5, 4/6, 5/7, 6/8}
    \node[cell] (fig) at (1+2*\di, -1-2*\r) {\includegraphics[width=\xgridx,keepaspectratio]{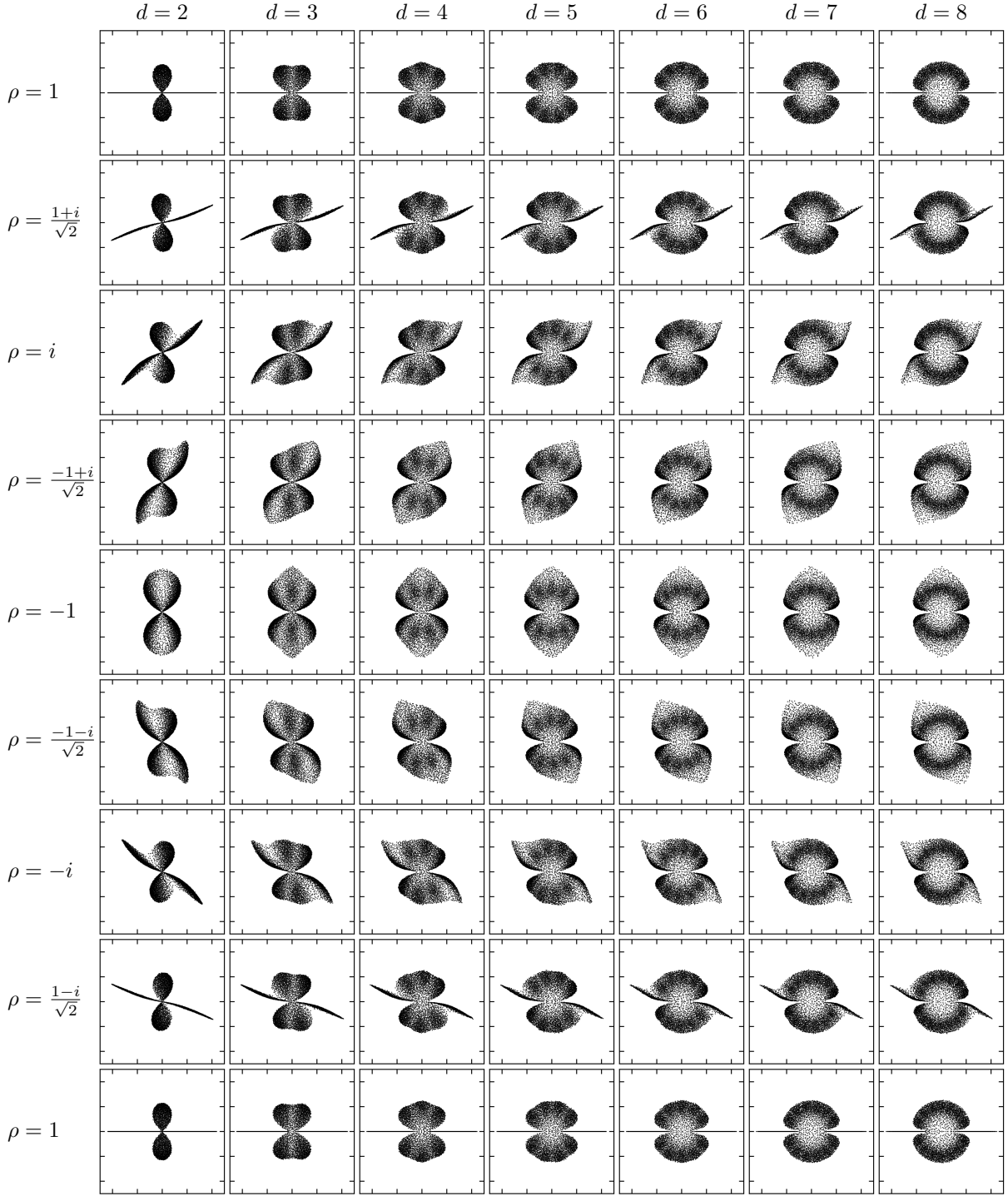}};
}
\end{tikzpicture}
\caption{$C=\sqrt{\rho}I$. This corresponds to each $A_j$ being independent and elliptic with variance 1 and correlation $\rho$, for $\abs\rho=1$.}
\label{a49}
\end{figure}

\newlength{\cellwidth}

\newcommand{\ninepanel}[2]{
\begin{figure}[H]
    \centering
    \setlength{\cellwidth}{\dimexpr\linewidth/7\relax}
    \begin{tikzpicture}[
        x=\cellwidth,
        y=1cm,
        collabel/.style={
            anchor=center,
            inner sep=0pt
        }
    ]
        \foreach \a/\aa in {
            0/{0},
            1/{1/6},
            2/{1/3},
            3/{1/2},
            4/{2/3},
            5/{5/6},
            6/{1}
        }{%
            \node[
                anchor=north,
                inner sep=0pt
            ] at ({\a+0.5},0) {%
                \includegraphics[
                    width=\cellwidth,
                    keepaspectratio
                ]{#1_\a_o_6.png}%
            };
            \node[collabel] at ({\a+0.5},0.2) {$a=\aa$};
        }

    \end{tikzpicture}
    \captionsetup{skip=4pt}
    #2
    \label{#1}
\end{figure}
}

\ninepanel{10_r}{
    \caption{$C=\pmat{1 & 0}^{\top}
    \pmat{1 & ai & 0 & \cdots & 0}$}
}

\ninepanel{100_r}{
    \caption{$C=\pmat{1 & 0 & 0}^{\top}
    \pmat{1 & ai & 0 & \cdots & 0}$}
}

\ninepanel{11_r}{
    \caption{$C=\pmat{1 & 1}^{\top}
    \pmat{1 & ai & 0 & \cdots & 0}$}
}

\ninepanel{ii_r}{
    \caption{$C=\pmat{i & i}^{\top}
    \pmat{1 & ai & 0 & \cdots & 0}$}
}

\ninepanel{iii_r}{
    \caption{$C=\pmat{i & i & i}^{\top}
    \pmat{1 & ai & 0 & \cdots & 0}$}
}

\ninepanel{1ii_r}{
    \caption{$C=\pmat{1 & i & i}^{\top}
    \pmat{1 & ai & 0 & \cdots & 0}$}
}

\ninepanel{10i_r}{
    \caption{$C=\pmat{1 & 0 & i}^{\top}
    \pmat{1 & ai & 0 & \cdots & 0}$}
}

\ninepanel{i1i_r}{
    \caption{$C=\pmat{i & 1 & i}^{\top}
    \pmat{1 & ai & 0 & \cdots & 0}$}
}

\begin{figure}[H]
    \centering
    \setlength{\cellwidth}{\dimexpr\linewidth/7\relax}
    \begin{tikzpicture}[
        x=\cellwidth,
        y=1cm,
        collabel/.style={
            anchor=center,
            inner sep=0pt
        }
    ]
        \foreach \a/\aa in {
            0/{2},
            1/{3},
            2/{4},
            3/{5},
            4/{6},
            5/{7},
            6/{8}}{%
            \node[
                anchor=north,
                inner sep=0pt
            ] at ({\a+0.5},0) {%
                \includegraphics[
                    width=\cellwidth,
                    keepaspectratio
                ]{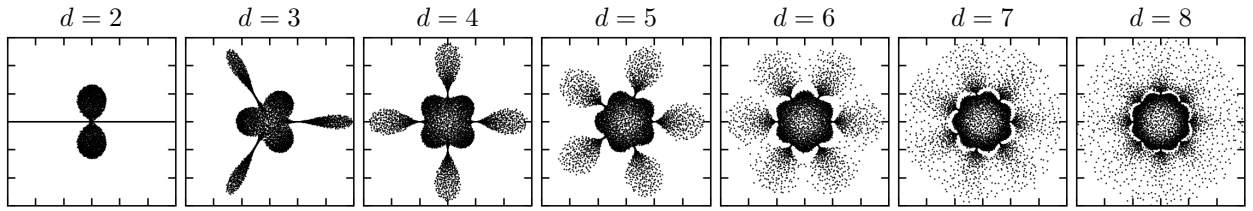}%
            };
            \node[collabel] at ({\a+0.5},0.2) {$d=\aa$};
        }

    \end{tikzpicture}
    \captionsetup{skip=4pt}
    \caption{$C=\bmat{F_d&0}$ where $F_d$ the $d\times d$ unitary discrete Fourier transform matrix.}
\end{figure}

\begin{figure}[H]
    \centering
    \setlength{\cellwidth}{\dimexpr\linewidth/7\relax}
    \begin{tikzpicture}[
        x=\cellwidth,
        y=1cm,
        collabel/.style={
            anchor=center,
            inner sep=0pt
        }
    ]
        \foreach \a/\aa in {
            0/{2},
            1/{3},
            2/{4},
            3/{5},
            4/{6},
            5/{7},
            6/{8}}{%
            \node[
                anchor=north,
                inner sep=0pt
            ] at ({\a+0.5},0) {%
                \includegraphics[
                    width=\cellwidth,
                    keepaspectratio
                ]{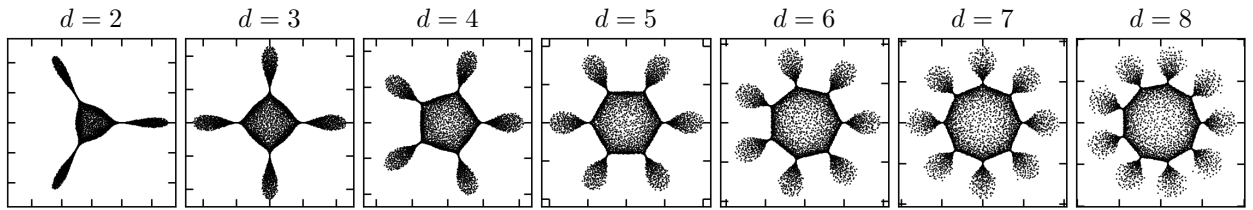}%
            };
            \node[collabel] at ({\a+0.5},0.2) {$d=\aa$};
        }

    \end{tikzpicture}
    \captionsetup{skip=4pt}
    \caption{$C=\bmat{I&iP\\P&-iI}$ for even $d$ where $P$ is the permutation which reverses coordinate order. For odd $d$, insert a row and column with a $\sqrt2$ in the center in between the blocks. In particular, $\frac12CC^*=I$ and $\frac12CC^\top$ is a larger version of $P$.}
\end{figure}

\newcommand{\onepanel}[3]{
\begin{subfigure}[t]{#3\textwidth}
    \centering
    \includegraphics[width=\linewidth]{#1.png}
    \caption{#2}
\end{subfigure}}

\begin{figure}[H]
    \centering
    \onepanel{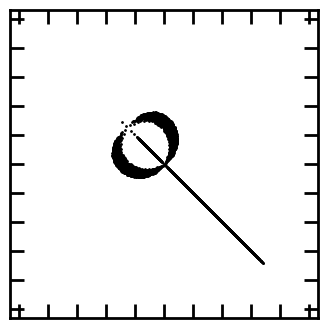}{$C=\pmat{2i & 2i\\-1+i&-1+i}$}{.32}
    \hfill
    \onepanel{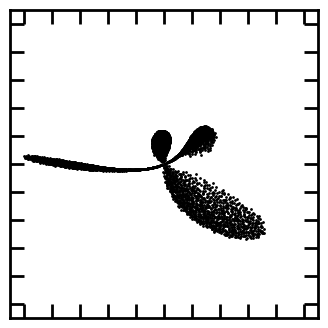}{$C=\pmat{-2-2i & -1-i\\1-i&2}$}{.32}
    \hfill
    \onepanel{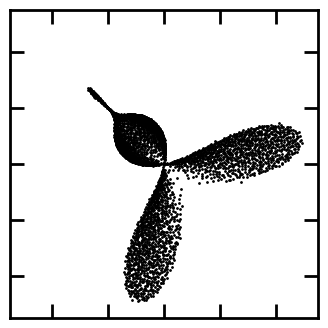}{$C=\pmat{-1&1\\-i&1}$}{.32}
\end{figure}

\begin{figure}[H]
    \centering
    \onepanel{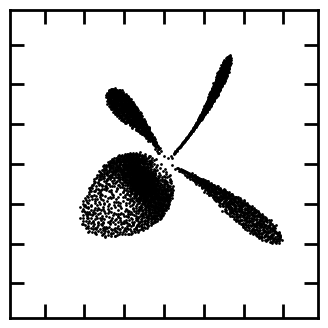}{$C=\pmat{1-i & -2+2i & -1+i\\-2&-1+3i&i\\1&0&i}$}{.32}
    \hfill
    \onepanel{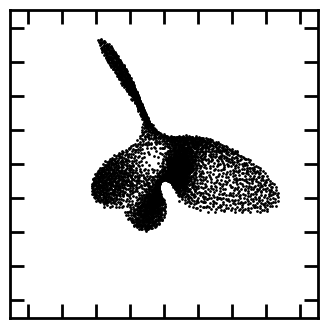}{$C=\pmat{-i & -2 & 1\\-1&3+i&-i\\-i&1-i&1}$}{.32}
    \hfill
    \onepanel{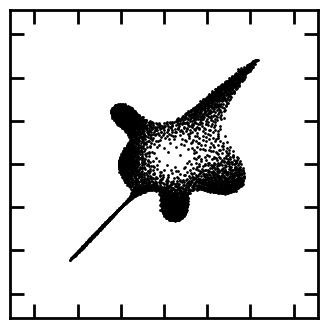}{$C=\pmat{1-i&-2&1-i\\0&-1+i&-i\\-1-i&0&0}$}{.32}
\end{figure}

\begin{figure}[H]
    \centering
    \onepanel{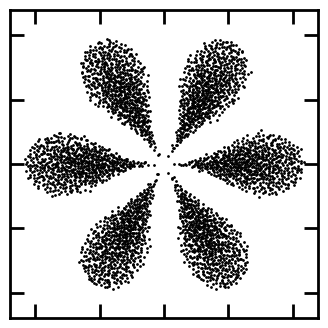}{$C=\pmat{2&2&2\\0&0&0\\-e^{\frac43\pi i}&-e^{\frac23\pi i}&-1}$}{.32}
    \hfill
    \hfill
    \onepanel{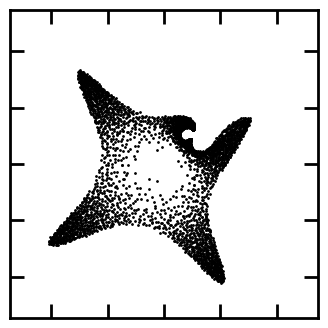}{$C=\pmat{-i&-i&1-i\\1+i&1+i&2i\\0&0&0}$}{.32}
    \hfill
    \onepanel{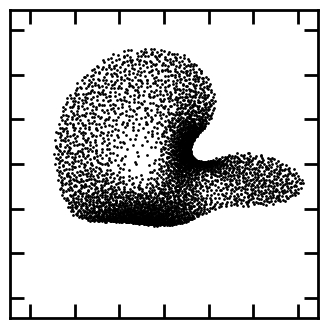}{$C=\pmat{2&1&-2i\\-2&1-i&2i\\i&1&1}$}{.32}
\end{figure}

\section*{AI Usage Statement}
The majority of this work was completed by April of 2026, without the use if LLMs. LLMs were subsequently used to check the results and search for related works and references. Several errors were identified via LLMs and manually corrected. LLM output identified that the covariance matrix appearing in \Cref{a46} was of a form that had been studied. The references it provided led to the simplification of the proof of \Cref{a31}. LLMs were used to generate a first draft of the abstract, though the version finally appearing has been heavily edited. LLMs were used to produce example TikZ figures sufficiently similar for our intentions for easy manual adaptation.

\bibliographystyle{alpha}
\bibliography{outbib}

\end{document}